\documentclass[a4paper, oneside,11 pt]{article}

\makeatletter
\let\@fnsymbol\@arabic
\makeatother

\usepackage[latin1]{inputenc}
\usepackage[T1]{fontenc}
\usepackage[english]{babel}

\usepackage[all]{xy}
\usepackage{verbatim}
\usepackage{amsmath}
\usepackage{amsthm}
\usepackage{graphicx}
\usepackage{amssymb}
\usepackage{epstopdf}
\usepackage{url}
\usepackage{amsfonts}
\usepackage{todonotes}
\newcommand{\mc}{\mathcal}
\newcommand{\pt}{\partial}

\newcommand{\vp}{\varphi}
\newcommand{\nn}{\nonumber}

\newcommand{\br}{\mathbb{R}}

\newcommand{\te}{\theta}

\newcommand{\e}{\varepsilon}
\newcommand{\x}{\bar{x}}
\renewcommand{\r}{\rho}

\renewcommand{\(}{\left(}
\renewcommand{\)}{\right)}
\renewcommand{\[}{\left[}
\renewcommand{\]}{\right]}

\newcommand{\ra}{\rightarrow}
\newtheorem{thm}{Theorem}
\newtheorem{lem}[thm]{Lemma}
\newtheorem{cor}[thm]{Corollary}
\newtheorem{prop}[thm]{Proposition}

\def\be{\begin{equation}}
\def\ee{\end{equation}}
\def\bea{\begin{eqnarray}}
\def\eea{\end{eqnarray}}

\numberwithin{thm}{section}
\numberwithin{equation}{section}

\title{A Gradient Flow Approach\\ to Quantization of Measures}

\author{Emanuele Caglioti
\thanks{University of Roma La Sapienza, Dipartimento Guido Castelnuovo, Piazzale Aldo Moro 5, 00185 Roma, ITALY. Email: \textsf{caglioti@mat.uniroma1.it}}
\and
Fran\c cois Golse\
\thanks{Ecole polytechnique, Centre de math\'ematiques Laurent Schwartz, 91128 Palaiseau Cedex, FRANCE. email: \textsf{francois.golse@polytechnique.edu}}
\and
Mikaela Iacobelli\
\thanks{University of Roma La Sapienza, Dipartimento Guido Castelnuovo, Piazzale Aldo Moro 5, 00185 Roma, ITALY. Email: \textsf{iacobelli@mat.uniroma1.it}}
\
\thanks{Ecole polytechnique, Centre de math\'ematiques Laurent Schwartz, 91128 Palaiseau Cedex, FRANCE. email: \textsf{mikaela.iacobelli@polytechnique.edu}}
}

\date{}

\begin{document}

\maketitle

\begin{abstract}
In this paper we study a gradient flow approach to the problem of quantization  of measures in one dimension.
By embedding our problem in $L^2$, we find a continuous version of it that corresponds to the limit
as the number of particles tends to infinity. Under some suitable regularity assumptions on the density,
we prove uniform stability and quantitative convergence result for the discrete and continuous dynamics. 
\end{abstract}

\tableofcontents

\section{Introduction}

\subsection*{The quantization problem in the static case}

The problem of quantization of a $d$-dimension probability distribution  by discrete probabilities with a given number of points can be stated as follows: Given a probability density  $\r$, approximate it  by a convex combination of a finite number $N$ of Dirac masses. The quality of the approximation is usually measured in terms of the Monge-Kantorovich metric.
Much of the early attention in the engineering and statistical literature was concentrated on the one-dimensional quantization problem. This problem arises in several contexts and has applications in information theory (signal compression), cluster analysis (quantization of empirical measures), pattern recognition,
speech recognition, numerical integration, stochastic processes (sampling design), mathematical models in economics (optimal location of service centers), and kinetic theory. For a detailed exposition and a complete list of references, we refer to the monograph \cite{GL}.

We now introduce the setup of the problem.
Fixed $r\ge 1$, consider $\r$ a probability density on an open set $\Omega \subset \br^d$ such that
$$
\int_{\Omega}|y|^r \rho(y)dy<\infty.
$$
Given $N$ points $x^{1}, \ldots, x^{N} \in \Omega,$ one wants to find the best approximation of $\rho,$ in the sense of Monge-Kantorovich, by a convex combination of Dirac masses centered at $x^{1}, \ldots, x^{N}.$ Hence one minimizes
$$
\inf \bigg\{ MK_r\bigg(\sum_im_i \delta_{x^i}, \r(y)dy\bigg)\,:\, m_1, \ldots, m_N\ge0, \  \sum_im_i=1
\bigg\},
$$
with
$$
MK_r(\mu, \nu):=\inf\bigg\{
\int_{\Omega\times\Omega}|x-y|^rd\gamma(x,y)\,:\,
(\pi_1)_\#\gamma=\mu,\  (\pi_2)_\#\gamma=\nu\bigg\},
$$
where $\gamma$ varies among all probability measures on $\Omega \times \Omega$, and $\pi_i: \Omega \times \Omega \to \Omega$ ($i=1,2$) denotes the canonical projection onto the $i$-th factor
 (see \cite{AGS,V} for more details on the Monge-Kantorovitch distance between probability measures).

As shown for instance in \cite[Chapter 1, Lemmas 3.1\ and 3.4]{GL},  the following facts hold:
\begin{enumerate}
\item
The best choice of the masses $m_i$ is given by
$$
m_i:= \int_{W(\{x^{1}, \ldots, x^{N}\} | x^i)} \r(y)dy,
$$
where 
$$
W(\{x^{1}, \ldots, x^{N}\}|x^i):= \{y \in \Omega\ : \ |y-x^i| \le |y-x^{j}|,\  j \in 1, \ldots, N \}
$$ 
is the so called  \emph{Voronoi cell} of $x^i$ in the set $x^1, \ldots, x^N.$
\item The following identity holds:
\begin{multline*}
\inf \bigg\{ MK_r\bigg(\sum_im_i \delta_{x^i}, \r(y)dy\bigg)\,:\, m_1, \ldots, m_N\ge0, \  \sum_im_i=1
\bigg\}\\
=F_{N,r} (x^{1}, \ldots, x^{N}) ,
\end{multline*}
where
 $$
F_{N,r} (x^{1}, \ldots, x^{N}) := \int_\Omega \underset{1\le i \le N}{\mbox{min}} | x^i-y |^r\rho(y)dy.
$$
\end{enumerate}
If one chooses $x^{1}, \ldots, x^{N}$ in an optimal way by minimizing the functional $F_{N,r}: (\br^d)^N \ra \br^+,$
in the limit as $N \ra \infty,$ these points distribute themselves accordingly to a probability density proportional to $\rho^{d/{d+r}}. $ In other words, by \cite[Chapter 2, Theorem 7.5]{GL} one has
\be\label{close}
\frac{1}{N}\sum_{i=1}^N \delta_{x^i} \rightharpoonup \frac{\rho^{d/{d+r}}}{\int_\Omega \rho^{d/{d+r}}(y)dy}\,dx.
\ee
These issues are relatively well understood from the point of view of the calculus of variations \cite[Chapter 1, Chapter 2]{GL}.
Our goal here is to consider instead a dynamic approach to this problem, as we shall describe now.\\

\subsection*{A dynamical approach to the quantization problem}

Given $N$ points $x^{1}_0, \ldots, x^{N}_0$, we consider their evolution  under the gradient flow
generated by $F_{N,r}$, that is, we solve the system of ODEs in $(\br^d)^N$
\be
\label{eq:ODEintro}
\left\{
\begin{array}{l}
\bigl(\dot x^{1}(t),\ldots,\dot x^{N}(t)\bigr)=-\nabla F_{N,r}\bigl(x^{1}(t),\ldots,x^{N}(t)\bigr),\\
\bigl(x^{1}(0),\ldots,x^{N}(0)\bigr)=(x^{1}_0, \ldots, x^{N}_0)
\end{array}
\right.
\ee


As usual in gradient flow theory, as $t \to \infty$ one expects that the points $\bigl(x^{1}(t),\ldots,x^{N}(t)\bigr)$ converge to a minimizer $(\bar x^{1},\ldots,\bar x^N)$ of $F_{N,r}.$ Hence (in view of \eqref{close}) the empirical measure 
$$
\frac{1}{N}\sum_{i=1}^N \delta_{\bar x^i}
$$
is expected to converge to 
$$
\frac{\rho^{d/{d+r}}}{\displaystyle\int_\Omega \rho^{d/{d+r}}(y)dy}dx
$$
as $N \to \infty$.

We now seek to pass to the limit in the ODE above as $N \to \infty$. For this, we take a set of reference points $(\hat x^{1},\ldots,\hat x^N)$ and we parameterize a general family of $N$ points $x^i$ as the image of $\hat x^i$ via 
a smooth map $X:\br^d\to \br^d$, that is
$$
x^i=X(\hat x^i).
$$
In this way, the function $F_{N,r}(x^{1},\ldots,x^{N})$ can be rewritten in terms of the map $X$ and (a suitable renormalization of it) should converge to a functional $\mathcal F[X]$. Hence, we expect that the evolution of $x^i(t)$ 
for $N$ large is well-approximated by the $L^2$-gradient flow of $\mathcal F$.

Although this formal argument may look convincing, already the $1$ dimensional case is nontrivial, and will be studied in detail in the present paper. The higher dimensional case is much harder. For one thing, even in space 
dimension $2$, there is no obvious analogue of the functional $\mathcal F[X]$ in the continuous limit. The present paper is focussed on the $1$ dimensional setting, and the higher dimensional case is left for future work.

\subsection*{The 1D case}

With no loss of generality we take $\Omega$ to be the open interval $(0,1)$
and we consider $\r$ a smooth probability density on $\Omega.$
In order to obtain a continuous version of the functional
$$
F_{N,r}(x^{1}, \ldots, x^{N})=\int_{0}^1\underset{1\le i \le N}{\mbox{min}} | x^i-y |^r\rho(y)\,dy,
$$ 
with $0\le x^{1}\le \ldots \le x^{N}\leq 1$, assume that
$$
x^i=X\bigg(\frac{i-1/2}{N}\bigg), \qquad i=1, \ldots, N
$$
with $X: [0,1] \to [0,1]$ a smooth non-decreasing map such that $X(0)=0$ and $X(1)=1$. Then,
as explained in Appendix \ref{app:discrete cont},
$$
N^rF_{N,r}(x^{1}, \ldots, x^{N}) \longrightarrow C_r\int_0^1 \r(X(\te))|\pt_\te X(\te)|^{r+1}d\te:=\mc{F}[X]
$$
as $N\ra \infty,$ where $C_r:=\frac{1}{2^r(r+1)}.$

By a standard computation \cite{E} we obtain the gradient flow PDE for $\mc{F}$ for the $L^2$-metric,
\begin{multline}\label{gradient flow}
\pt_tX(t,\te)=C_r\Big((r+1)\pt_\te\big(\r(X(t,\te))|\pt_\te X(t,\te)|^{r-1}\pt_\te X(t,\te)\big)\\
-\r'(X(t,\te))|\pt_\te X(t,\te)|^{r+1} \Big),
\end{multline}
completed with the Dirichlet boundary condition
\begin{equation}
\label{eq:boundary}
X(t,0)=0, \qquad X(t,1)=1.
\end{equation}
Let us notice that, in the particular case $\r \equiv 1,$  \eqref{gradient flow} becomes a $p$-Laplacian equation
$$
\pt_tX=C_r(r+1)\pt_\te\big(|\pt_\te X|^{r-1}\pt_\te X\big)
$$
with $p-1=r$ (see \cite{D, Va} and references therein for a general treatment of this class of equations).\\

\subsection*{From the Lagrangian to the Eulerian setting}

Equation \eqref{gradient flow} provides a Lagrangian description of the evolution of our system of particles in the limit $N \to \infty.$  We can also study the Eulerian picture for the gradient flow PDE.
If we denote by $f(t,x)$ the image of the Lebesgue measure through the map $X,$ i.e.
$$
f(t,x)dx=X(t,\te)_\#d\te,
$$
then the PDE satisfied by $f$ takes the form \footnote{Indeed since $\pt_t X=b(t, X)$ with $b(t,y):=C_rr\(\frac{\r(y)}{f(t,y)^{r+1}} \)$ (this follows by a direct computation starting from \eqref{gradient flow}),
 the function $f\equiv f(t,x)$ solves the continuity equation $\pt_t f(t,x)+{\rm div}(b(t,x)f(t,x))=0,$ as shown for instance in \cite{A}.}
\be\label{eulerian}
\pt_tf(t,x)=-rC_r\pt_x\bigg(f(t,x)\pt_x\Big(\frac{\r(x)}{f(t,x)^{r+1}}\Big) \bigg),
\ee
with periodic boundary conditions,
and we expect the following long time behavior
$$
f(t,x) \longrightarrow \frac{\rho^{1/(r+1)}(x)}{\int_0^1 \r(y)^{1/(r+1)}dy}
\qquad \text{as $t\to \infty$}.
$$
Notice that if $\r \equiv 1 ,$ \eqref{eulerian} becomes
$$
\pt_tf=-C_r(r+1)\pt_x^2\big(f^{-r}\big),
$$
which is an equation of very fast diffusion type \cite{BV,Va,Va2}.
It is interesting to point out that the above equation set on the whole space $\mathbb R$ or with zero Dirichlet boundary conditions has no solutions, since all the mass instantaneously disappear \cite[Theorem 3.1]{Vazquez}.
It is therefore crucial that in our setting the equation has periodic boundary conditions. In particular, as we shall see, our equation satisfies a comparison principle (see Lemma \ref{lem:comparison}).\\

\subsection*{Assumptions on $\rho$ and convexity of the functionals}

Notice that our heuristic arguments in the previous section were based on the assumption that both the 
gradient flows of $F_{N,2}$ and of $\mathcal F$ converge to a minimizer as $t \to \infty$. Of course this is true if $F_{N,2}$ and $\mathcal F$ are convex
 \cite{AGS}.
Actually, notice that
we are trying to show that the limits as $N \to \infty$ and $t\to \infty$ commute,
and for this we need to prove that the discrete and the continuous gradient flows
remain close in the $L^2$ sense, uniformly with respect to $t$.
Therefore, the convexity of $\mathcal F$ and $F_{N,2}$ seems to be a very natural issue for the validity of our gradient flow strategy.

As shown in Appendix \ref{app:hessian}, for the hessian of $\mathcal F$
to be nonnegative at ``points'' $X$ which are Lipschitz and uniformly monotone, one has to assume $\rho$ to be sufficiently close to a constant in $C^2$.
We shall therefore adopt this condition on $\rho$.

Whether this condition on $\rho$ ensures that $F_{N,2}$ is also convex is left undecided. Nevertheless  we are able to prove that the discrete flow and the continuous one remain close by a combination of arguments including the maximum principle and $L^2$-stability (see Section \ref{sect:rho neq 1}).\\

\subsection*{Statement of the results}

In order to simplify our presentation,
in the whole paper we shall focus only on the case $r=2$.
Indeed, this has the main advantage of 
simplifying some of the computations
allowing us to highlight the main ideas.
As  will be clear from the sequel, this case already incorporates all the main features and
difficulties of the problem, and this specific choice does not play any essential role.

As we mentioned in the previous section, the properties of $\rho$ are crucial in the proofs.
Notice also that \eqref{gradient flow} is of $p$-laplacian type,
which is a degenerate parabolic equation. In order to avoid degeneracy,
it is necessary for the solution to be an increasing function of $\theta$.
For this reason, we assume this on the initial datum and prove that this monotonicity is preserved along the flow.

It is worth noticing that 
the monotonicity estimate at the discrete level says that if  $x^{i+1}(0)-x^i(0) \approx \frac{1}{N}$ for all $i,$
this property is preserved in time (up
to multiplicative constants). In particular the points $\{x^i(t)\}_{i=1,\ldots,N}$ can never collide.

Our main result shows that, under the two above mentioned assumptions
(that is, $\rho$ is close to a constant in $C^2$ and the initial datum is smooth 
and increasing) the discrete and the continuous gradient flows
remain uniformly close in $L^2$ for {\em all} times.
Notice however that the results in the case $\rho\equiv 1$ and $\rho\not\equiv 1$
are quite different.
Indeed, when $\rho \equiv 1$ the equation \eqref{gradient flow} depends on $\partial_\theta X$
and $\partial_{\theta\theta}X$, but not on $X$ itself. This fact plays a role in several places,
both for showing the monotonicity of solutions (in particular for the discrete case)
and in the convergence estimate.
In particular, while in the case $\rho\equiv 1$ we obtain convergence of the discrete flow to the continuous one
for all initial data, the case $\rho\not\equiv 1$ requires an additional assumption at time $0$ (see \eqref{eq:close0}).

One further comment concerns the time scaling: notice that, in order to obtain a nontrivial
limit of our functional $F_{N,r}$, we needed to rescale them by $1/N^r$.
In addition to this, since we want to compare gradient flows, we have to take into account
that the Euclidean metric in $\mathbb R^N$ has to be rescaled by a factor $1/N$ to be compared with the $L^2$
norm.\footnote{
Let $\bar x:=(x^1,\ldots,x^N), \bar y:=(y^1,\ldots,,y^N) \in \mathbb R^N$, and embed 
these points into $L^2([0,1])$ by defining the functions
$$
X(\theta):=x^i,\quad Y(\theta):=y^i,\qquad \forall\, \theta \in \biggl(\frac{i-1}{N},\frac{i}N\biggr).
$$
Then $|\bar x-\bar y|^2=\sum_{i=1}^N|x^i-y^i|^2$ while $\|X-Y\|_{L^2}^2=\frac{1}{N}\sum_{i=1}^N|x^i-y^i|^2.$
}
Hence, to compare the discrete and the continuous gradient flows,
we need to rescale the former in time by a factor $N^{r+1}$.

We now state our convergence results, first when $\rho\equiv 1$
and then for the general case.
It is worth to point out 
that the best way to approximate the uniform measure on $[0,1]$
with the sum of $N$ Dirac masses it to put masses of size $1/N$ centered at points $(i-1/2)/N$
and 
$$
MK_1\biggl(\frac{1}{N}\sum_{i=1}^N\delta_{(i-1/2)/N},d\theta\biggr)=\frac{1}{4N}
$$
(see the computation in the proof of Theorem \ref{thm:Euler 1}). Hence the result in our next theorem
shows that the gradient flow approach provides, for $N$ and $t$ large, the best approximation rate.

\begin{thm}
\label{thm:main1}
Let $\rho\equiv 1$, 
$\big(x^1(t),\ldots,x^N(t)\bigr)$ the gradient flow of $F_{N,2}$,
and $X(t)$ the gradient flow of $\mathcal F$ starting from $X_0$.
Assume that $X_0 \in C^{4,\alpha}([0,1])$ and that
 there exist positive constants $c_0,C_0$ such that
$$
\frac{c_0}{N}\leq \x^i(0) - \x^{i-1}(0)\leq \frac{C_0}{N},\  \text{and} \ 
c_0  \leq \partial_\theta X_0\leq C_0.
$$
Define $X^i(t):=X\left(t,\frac{i-1/2}{N}\right)$, $\x^i(t):=x^i(N^3t)$, and $\mu_t^N:=\frac1N\sum_i \delta_{x^i(t)}$
Then there exist two constants $c',C'>0$, depending only on $c_0$, $C_0$, and $\|X_0\|_{C^{4,\alpha}([0,1])}$,
such that, for all $t \geq 0$,
$$
\frac{1}{N} \sum_{i=1}^N \(\x^i(t)-X^i(t)\)^2
\leq e^{-c' t} \frac{1}{N} \sum_{i=1}^N \(\x^i(0)-X^i(0)\)^2
+\frac{C'}{N^4}
$$
and
$$
MK_1(\mu_t^N,d\theta) \leq \frac{1}{4N}+C' \,e^{-c't/N^3}+\frac{C'}{N^2}.
$$
In particular
$$
MK_1(\mu_t^N,d\theta) \leq \frac{1}{4N}+\frac{2C'}{N^2} \qquad \forall\,t \geq \frac{2N^3\log N}{c'}.
$$
\end{thm}

\begin{thm}
\label{thm:main2}
Let 
$\big(x^1(t),\ldots,x^N(t)\bigr)$ be the gradient flow of $F_{N,2}$,
and $X(t)$ the gradient flow of $\mathcal F$ starting from $X_0$.
Assume that $X_0 \in C^{4,\alpha}([0,1])$ for some $\alpha>0$ and that
 there exist two positive constants $c_0,C_0$ such that
$$
\frac{c_0}{N}\leq x^i(0) - x^{i-1}(0)\leq \frac{C_0}{N}, \  \text{and} \ 
c_0  \leq \partial_\theta X_0\leq C_0.
$$
 Define $X^i(t):=X\left(t,\frac{i-1/2}{N}\right)$, $\x^i(t):=x^i(N^3t)$, and $\mu_t^N:=\frac1N\sum_i \delta_{x^i(t)}$,
 and assume that $\rho:[0,1]\to (0,\infty)$ is a periodic probability density of class $C^{3,\alpha}$ with $\|\rho'\|_\infty+\|\rho''\|_\infty \leq \bar \e$
 and that
\be
\label{eq:close0}
|X^i(0) - x^i(0)| \leq \frac{\bar C}{N^2} \qquad \forall\,i=1,\ldots,N.
\ee
for some positive constants $\bar \e,\bar C$.
Then there exist two constants $c',C'>0$, depending only on $c_0$, $C_0$, $\|\rho\|_{C^{3,\alpha}([0,1])}$
and $\|X_0\|_{C^{4,\alpha}([0,1])}$,
such that the following holds:
if $\bar \e$ is small enough (in terms of $c_0$, $C_0$, and $\bar C$) we have
$$
\frac{1}{N} \sum_{i=1}^N \(\x^i(t)-X^i(t)\)^2
\leq \frac{C'}{N^4}\qquad \text{for all} \ \,t \geq 0
$$
and 
$$
MK_1(\mu_t^N,\gamma \rho^{1/3}\,d\theta) \leq  C' \,e^{-c't/N^3}+\frac{ C'}{N} \qquad \text{for all} \ \,t \geq 0,
$$
where 
$$
\frac{1}{\gamma}:=\int _0^1\rho(\theta)^{1/3}\,d\theta.
$$
In particular
$$
MK_1(\mu_t^N,\gamma \rho^{1/3}\,d\theta) \leq \frac{ C'}{N} \qquad \text{for all} \ \,t \geq \frac{N^3\log N}{ c'}.
$$
\end{thm}

As a consequence of our results, under the assumption that $\rho$ is $C^2$ close to $1$
we obtain a quantitative version of the results in \cite{GL}:

\begin{cor}There exist two constants $\bar \e>0$ and $C>0$ such that the following holds:
assume that $\|\rho'\|_\infty+\|\rho''\|_\infty \leq \bar \e$, and let $(x^1,\ldots,x^N)$ be a minimizer of $F_{N,2}$.
Then 
$$
MK_1(\mu^N,\gamma \rho^{1/3}\,d\theta)\leq \frac{C}{N}
$$
where 
$$
\mu^N:=\frac1N\sum_i \delta_{x^i}
$$
and 
$$
\frac{1}{\gamma}:=\int _0^1\rho(\theta)^{1/3}\,d\theta.
$$
\end{cor}

The paper is structured as follows: in the next section
we collect several preliminary results both on the discrete and the continuous gradient flow.
Then, we prove the convergence result first in the case $\rho\equiv 1$,
and finally in the case $\|\rho -1\|_{C^2([0,1])} \ll 1$.

In the whole paper we assume that $0<\lambda \leq \rho \leq 1/\lambda$.

\section{Preliminary results}

\subsection{The discrete gradient flow}
We begin by computing the discrete gradient flow: 
as shown in the appendix, given points $0\le x^{1}\le \ldots \le x^{N}\le 1$,
one has
$$
F_{N,r}(x^{1}, \ldots, x^{N})=
 \sum_{i=1}^N \int_{x^{i-1/2}}^{x^{i+1/2}} |y-x^i|^r\r(y)dy
 $$
where 
$$
x^{i+1/2}:=\frac{x^i+x^{i+1}}{2} \qquad  \forall\,i=2,\ldots,N-1,
$$ 
while we set $x^{1/2}:=0$ and $x^{N+1/2}:=1$.
Then, a direct computation gives
\be
\label{eq:gradient FN2}
\frac{\pt F_{N,2}}{\pt{x^i}}(x^1, \ldots, x^N)= - 2\int_{x^{i-1/2}}^{x^{i+1/2}} (y-x^i)\r(y)dy.
\ee
Moreover, assuming that $\rho$ is at least of class $C^0$ it is easy to check that $\nabla F_{N,2}$
is bounded and continuously differentiable, hence $F_{N,2}$ is of class $C^2$.
Thus the gradient flow of $F_{N,2}$ is unique and exists globally for all $t\ge0$ 
by the Cauchy-Lipschitz Theorem for ODEs.

\subsection{The continuous gradient flow }

In order to construct a solution to the continuous gradient flow \eqref{eq:ODE X}
we start from the Eulerian description that we look as a PDE on $[0,1]$ with periodic boundary conditions.

\subsubsection{The Eulerian flow}

Recall that by assumption $\lambda \leq \rho\leq 1/\lambda$ for some $\lambda >0$.
Given $f(t,x)$ a solution of \eqref{eulerian},
we set
$$
m(x):=\rho(x)^{1/3},\qquad u(t,x):=\frac{f(t,x)}{m(x)}.
$$
With these new unknowns \eqref{eulerian} becomes
\begin{equation}
\label{eq:euler comparison}
\partial_tu = -\frac{1}{4\,m(x)} \,\partial_x\biggl(m(x)\partial_x\biggl(\frac{1}{u^2}\biggr) \biggr) \qquad
\text{on $[0,\infty)\times [0,1]$}
\end{equation}
with periodic boundary conditions.
The advantage of this form is double: first of all, the above PDE enjoys a comparison principle;
secondly, constants are solutions.
Since for our purposes, only comparison with constants is necessary, we will just show that.

\begin{lem}
\label{lem:comparison}
Let $u$ be a nonnegative solution of \eqref{eq:euler comparison} and $c$ be a positive constant.
Then both
$$
t \mapsto \int_0^1 (u-c)_-\,dx \quad \text{and}\quad t \mapsto \int_0^1 (u-c)_+\,dx
$$
are nonincreasing functions.
\end{lem}
\begin{proof}
We show just the first statement (the other being analogous).

Since constants are solutions of \eqref{eq:euler comparison}, it holds
$$
\partial_t(u-c) = -\frac{1}{4\,m} \,\partial_x\biggl(m\,\partial_x\biggl(\frac{1}{u^2} - \frac{1}{c^2}\biggr) \biggr).
$$
We now multiply the above equation by $-m\,\phi_\e\left(\frac{1}{u^2} - \frac{1}{c^2}\right)$,
with $\phi_\e$ a smooth approximation the indicator function of ${\mathbb R_+}$ satisfying $\phi_\e'\geq 0$. Integrating by parts we get
\begin{align*}
\frac{d}{dt}\int_0^1 \Psi_\e(u-c)\,dx&=
-\int_0^1 \phi_\e\left(\frac{1}{u^2} - \frac{1}{c^2}\right) \, \partial_t(u-c) \,m\,dx\\
&=-\frac{1}{4} \int_0^1 \biggl|\partial_x\biggl(\frac{1}{u^2} - \frac{1}{c^2}\biggr) \biggr|^2 \phi_\e'\biggl(\frac{1}{u^2} - \frac{1}{c^2}\biggr)\,m\,dx \leq 0,
\end{align*}
where we have set
$$
\Psi_\e(s):=-\int_0^s \phi_\e\biggl(\frac{1}{(\sigma+c)^2} - \frac{1}{c^2} \biggr)\,d\sigma.
$$
Letting $\e \to 0$ we see that $\Psi_\e(s) \to s_-$, hence
$$
\frac{d}{dt}\int_0^1 (u-c)_-\,dx\leq 0,
$$
proving the result.
\end{proof}

Thus, if $a_0 \leq u(0,x)\leq A_0,$ then $a_0 \leq u(t,x)\leq A_0$ for all $t \geq 0$.
We now apply this fact to show that if $f$ is bounded away from zero and infinity at the initial time, then
so it is for all positive times.
More precisely, recalling that by assumption $\lambda \leq\rho \leq \frac{1}{\lambda}$, we have
\begin{align*}
a_1 \leq f(0,x) \leq A_1 \quad &\Rightarrow \quad \lambda^{1/3}a_1 \leq u(0,x) \leq \frac{A_1}{\lambda^{1/3}}\\
&\Rightarrow \quad \lambda^{1/3}a_1 \leq u(t,x) \leq \frac{A_1}{\lambda^{1/3}}\\
&\Rightarrow \quad \lambda^{2/3}a_1 \leq f(t,x) \leq \frac{A_1}{\lambda^{2/3}}\qquad \forall\,t \geq 0.
\end{align*}

These a priori bounds show that \eqref{eulerian}
is a uniformly parabolic equation. In particular,
since $f$ is uniformly bounded for all times,
by parabolic regularity theory (see for instance \cite[Theorem $8.12.1$]{K}, \cite[Chapter 3, Section 3, Theorem 7]{F}, and \cite[Chapters 5, 6]{L}) we conclude that:
\begin{prop}
\label{prop:Euler}
Let $\lambda \in (0,1]$, and assume that $\rho:[0,1]\to [\lambda,1/\lambda]$ is periodic and of class $C^{k,\alpha}$ for some $k \geq 0$ and $\alpha \in (0,1)$.
Let $f(0,\cdot):[0,1]\to \mathbb R$ be a periodic function of class $C^{k,\alpha}$ satisfying $0<a_1 \leq f(0,\cdot) \leq A_1$, and let $f$ solve \eqref{eulerian} with periodic boundary conditions. Then
$$
 \lambda^{2/3}a_1 \leq f(t,x) \leq \frac{A_1}{\lambda^{2/3}}\qquad \text{for all} \ \,t \geq 0,
$$
$f(0,\cdot)$ is of class $C^{k,\alpha}$ for all $t \geq 0$,
and
there exists a constant $C$, depending only on
$\lambda$, $\|\rho\|_{C^{k,\alpha}}$, $k$, $\alpha$, $a_1$, and $A_1$, such that
$\|f(t,\cdot)\|_{C^{k,\alpha}([0,1])} \leq C$ for all $t \geq 0.$
\end{prop}

\subsubsection{The Lagrangian flow}
\label{sect:Lagr}
To obtain now existence and uniqueness for the gradient flow of $\mathcal F$, we simply define $X(t)$ for any $t \geq 0$ as the solution of the ODE (in $\theta$)
\be
\label{eq:ODE X}
\left\{
\begin{array}{ll}
\partial_\theta X(t,\theta)=\frac{1}{f(t,X(t,\theta))} & \text{on $[0,1]$},\\
X(t,0)=0,\\
\end{array}
\right.
\qquad \forall\, t\geq 0.
\ee
Notice that the boundary conditions $X(t,1)=1$ is automatically satisfied
since
$$
\int_0^{X(t,1)}f(t,x)\,dx=1
$$
and $f(t)>0$ is a probability on $[0,1]$.
Also,  notice that $X(t)$ has exactly one derivative more than $f(t)$.
Hence, by Proposition \ref{prop:Euler} we obtain:
\begin{prop}
\label{prop:Lagr}
Let $\lambda \in (0,1]$, and assume that $\rho:[0,1]\to [\lambda,1/\lambda]$ is periodic and of class $C^{k,\alpha}$ for some $k \geq 0$ and $\alpha \in (0,1)$.
Let $X(0,\cdot)$ satisfy $0<a_1 \leq \partial_\theta X(0,\cdot) \leq A_1$, $X(0,0)=1$, $X(0,1)=1$,
and $\|X(0,\cdot)\|_{C^{k+1,\alpha}([0,1])}<\infty$, and let $X(t,\cdot)$ solve \eqref{gradient flow}-\eqref{eq:boundary}. Then
$$
 \lambda^{2/3}a_1 \leq \partial_\theta X(t,\te) \leq \frac{A_1}{\lambda^{2/3}}\qquad \text{for all} \ \,t \geq 0,
$$
and there exists a constant $C$, depending only on
$\lambda$, $\|\rho\|_{C^{k,\alpha}}$, $k$, $\alpha$, $a_1$, and $A_1$, such that
$\|X(t, \cdot)\|_{C^{k+1,\alpha}([0,1])} \leq C$ for all $t \geq 0.$
\end{prop}

\section{The case $\rho\equiv 1$}
As we already mentioned we shall focus only on increasing initial data,
and as proved in the previous section this monotonicity is preserved in time,
hence $\partial_\theta X \geq 0$.

We first observe that, in the case $\rho\equiv 1$, the equation \eqref{gradient flow}
becomes
\be \label{eqsemplificata}
\pt_t X(t,\theta)=\frac14 \pt_{\theta}\(\pt_\theta X(t,\theta)^2\)= \frac12\pt_{\theta}X(t,\theta)\pt^2_{\te\te}X(t,\theta)
\ee
with Dirichlet boundary conditions \eqref{eq:boundary}.

\subsection{The $L^2$ estimate in the continuous case}

The following result shows the exponential stability in $L^2$ of the continuous
gradient flows.

\begin{prop}
\label{prop:L2 1}
Let $X_1, X_2$ be two solutions of \eqref{eqsemplificata} satisfying
\eqref{eq:boundary} and
\be
\label{mon_cond}
\pt_\theta X_i(0,\theta)\ge c>0,\qquad i=1,2.
\ee
Then
\be\nn
 \int_0^1|X_1(t,\theta)-X_2(t,\theta)|^2\,d\te \le \( \int_0^1|X_1(0,\theta)-X_2(0,\theta)|^2\,d\te \)e^{-4ct}.
\ee
\end{prop}
\begin{proof}
We first recall that the monotonicity condition \eqref{mon_cond} is preserved in time
(apply Proposition \ref{prop:Lagr} with $\lambda=1$).
Then, since $X_2-X_1$ vanishes at the boundary, one has
\begin{align*}
\frac{d}{dt} \int_0^1|X_1-X_2|^2\,d\te 
&=\int_0^1 (X_1-X_2)\,(\pt_\te(\pt_\te X_1^2)-\pt_\te(\pt_\te X_2^2))\,d\te \\ 
&=-\int_0^1( \pt_\te X_1-\pt_\te X_2)(\pt_\te X_1^2-\pt_\te X_2^2)\,d\te 	\\ 
&= - \int_0^1 ( \pt_\te X_1-\pt_\te X_2)^2(\pt_\te X_1+\pt_\te X_2)\,d\te .
\end{align*}
Using the monotonicity condition $\pt_\te X_i \geq c$ and the Poincar\'e inequality
on $[0,1]$ (see for instance Lemma \ref{lem:poincare}
and let $N \to \infty$), we get
\begin{align*}
- \int_0^1 ( \pt_\te X_1-\pt_\te X_2)^2(\pt_\te X_1+\pt_\te X_2)\,d\te 
&\le-2c \int_0^1 ( \pt_\te X_1-\pt_\te X_2)^2\,d\te\\ \nn
&\le-4c\int_0^1 ( X_1-X_2)^2\,d\te\\ \nn
\end{align*}
so that 
$$
\frac{d}{dt}\(e^{4ct} \int_0^1|X_1-X_2|^2(t, \theta)\, d\te\)\le 0.
$$
\end{proof}
This argument shows that, if at time zero  $X_1(0,\te)=X_2(0,\te)$ for a.e. $\theta \in (0,1),$ in particular $X_1(t,\te)=X_2(t,\te)$ for a.e. $\theta \in (0,1)$, for all $t\ge0$. Moreover if  $X_1(0,\te)-X_2(0,\te)$ is small in $L^2$ then it remains small in $L^2$ (continuity with respect to the initial datum),
and actually converges to zero exponentially fast.
In particular, noticing that $X(t,\theta)=\theta$ is a solution (corresponding to $f(t,x)=1$), we deduce that all solutions converge exponentially to it:
indeed, choosing $X_2(t,\theta)=\theta$ and assuming $c \leq 1$ we have
$$
 \int_0^1|X(t,\theta)-\te|^2\,d\te \le \( \int_0^1|X(0,\te)-\te|^2\,d\te \)e^{-4c t}.
$$

\subsection{Convergence of the gradient flows}
The functional $F_{N,2}(x^1, \ldots, x^{N})$ with $\rho \equiv 1$ is given by
\be
F_{N,2}(x^1, \ldots, x^{N})=\frac{|x^1|^3}{3}+\sum_{i=1}^{N-1} \frac{1}{12}|x^{i+1}-x^{i}|^3+\frac{|1-x^N|^3}{3},
\ee
hence the defining equation for the  gradient flow for $F_{N,2}$ is
\be
\dot{x}^i=-\frac{\pt F_{N,2}}{\pt x^i}=\frac14 \Big(\(x^{i+1}-x^i\)^2-\(x^{i}-x^{i-1}\)^2\Big) \qquad \text{for all} \ \,i=1,\ldots,N,
\ee
where by convention $x^0:=-x^1$ and $x^{N+1}:=2-x^N$.

The former convention comes from the following observation: in order to avoid problems at the boundary,
one could symmetrize the configuration of points $x^1,\ldots,x^N$ with respect to $0$ to get $N$ points $y^1,\ldots,y^N \in [-1,0]$ satisfying $y^i:=-x^i$.
By identifying $-1$ with $1$, we then get a family of $2N$ points on the circle where the dynamics is completely equivalent to ours. This means that, by adding $x^0$ and $x^{N+1}$ defined as above, we can see $x^1$ and $x^N$ as interior points. In the next section we will apply the same observation symmetrizing also the density
$\rho$ in the way described above.\\

In order to prove convergence, we want to find an equation for $X$ evaluated on the grid $(i-1/2)N$.

\begin{lem}
\label{lem:taylor 1}
Let $X(t,\theta)$ be a solution of \eqref{eqsemplificata}-\eqref{eq:boundary}  starting from an initial datum $X_0 \in C^{4,\alpha}([0,1])$ with $\partial_\theta X_0 \geq c_0>0$.
Let ${X}^i$ be the discretized solution defined at the points $(\frac{i-1/2}{N}, t)$, that is 
\be
\label{eq:Xi 1}
{X}^i(t):=X\biggl(\frac{i-1/2}{N},t\biggr)\qquad \text{for all} \ \,i=1,\ldots,N.
\ee
Then
$$
\pt_t {X}^i-N^3\frac{\pt F_{N,2}}{\pt{x^i}}(X^{1}, \ldots, X^{N})=R^i.
$$
with
$$
|R^i(t)| \leq \frac{\hat C}{N^2} \qquad \text{for all} \ \,t \geq 0,\ \  \text{for all} \ \,i=1,\ldots,N,
$$
where $\hat C$ depends only on $c_0$ and $\|X_0\|_{C^{4,\alpha}([0,1])}$.
\end{lem}

\begin{proof}
As we showed in Proposition \ref{prop:Lagr} we have $\partial_\theta X(t, \te)\geq c>0$ for all $t$, so that the equation \eqref{eqsemplificata}
remains uniformly parabolic and under our assumptions the solution $X(t)$ remains of class $C^4$ for all times, with
$$
\|X(t)\|_{C^4}\leq C\qquad \forall\,t \geq 0.
$$
By Taylor's expansion centered at $(\frac{i-1/2}{N}, t)$, one has
$$
X^{i+1}={X}^i+\frac{1}{N}\pt_\te {X}^i+\frac{1}{2N^2}\pt_{\te\te} {X}^i+
\frac{1}{6N^3}\pt_{\te\te\te}{X}^i +
O\biggl(\frac{\|X(t)\|_{C^4}}{N^4}\biggr),
$$
$$
X^{i-1}={X}^i-\frac{1}{N}\pt_\te {X}^i+\frac{1}{2N^2}\pt_{\te\te} {X}^i-\frac{1}{6N^3}\pt_{\te\te\te}{X}^i +
O\biggl(\frac{\|X(t)\|_{C^4}}{N^4}\biggr).
$$
Thus, with the convention $X^0:=-X^1$ and $X^{N+1}:=2-X^N$,
\begin{multline*}
\pt_t {X}^i-\frac{N^3}{4}\Big(\(X^{i+1}-X^{i}\)^2-\(X^{i}-X^{i-1}\)^2 \Big)=\\ \pt_t {X}^i-\frac{N^3}{4}\biggl[\frac{1}{N}\pt_\te {X}^i+\frac{1}{2N^2}\pt_{\te\te} {X}^i+
\frac{1}{6N^3}\pt_{\te\te\te}{X}^i +
O\biggl(\frac{\|X(t)\|_{C^4}}{N^4}\biggr)\biggr]^2\\+\frac{N^3}{4}\biggl[-\frac{1}{N}\pt_\te {X}^i+\frac{1}{2N^2}\pt_{\te\te} {X}^i-
\frac{1}{6N^3}\pt_{\te\te\te}{X}^i +
O\biggl(\frac{\|X(t)\|_{C^4}}{N^4}\biggr)\biggr]^2 ,
\end{multline*}
hence
\begin{multline*}
\pt_t {X}^i-\frac{N^3}{4}\Big(\(X^{i+1}-X^{i}\)^2-\(X^{i}-X^{i-1}\)^2 \Big)
\\
=\pt_t {X}^i-\frac{1}{2}\pt_\te {X}^i\pt_{\te\te} {X}^i+R^i=R^i,
\end{multline*}
with
$$
|R^i(t)| \leq C\, \frac{\|X(t)\|_{C^4}}{N^2}\leq \frac{\hat C}{N^2}.
$$
with $\hat C:=C\sup_{t \geq 0}\|X(t)\|_{C^4}$.
\end{proof}
In order to compare $X$ with $x^i$ we need to rescale times.
More precisely,
let us denote with ${\x}^i(t):=x^i(N^3t).$ Then
\be
\label{eq:xi 1}
\dot{\x}^i=\frac{N^3}{4}\Big(\(\x^{i+1}-{\x}^i\)^2-\(\x^{i}-\x^{i-1}\)^2\Big).
\ee
For simplicity of notation we set 
$$
W^i_X:=N\(X^{i+1}-{X}^i\),\quad
Y^i_X:=\big(W^i_X\big)^2,
$$
$$
W^i_{\x}:=N\(\x^{i+1}-{\x}^i\),\quad
Y^i_{\x}:=\big(W^i_{\x}\big)^2,
$$
(recall  the convention $X^0:=-X^1$ and $X^{N+1}:=2-X^N$).
The equation for ${X}^i$ can be written as
$$
\pt_t {X}^i=\frac{N}{4}\big(Y_{X}^i-Y_X^{i-1} \big)+R^i,
$$
while the equation for $W_{\x}^i$ (which follows easily from \eqref{eq:xi 1})
is given by
\be\label{eq.W}
\pt_t W_{\x}^i
=\frac{N^2}{4}\Big((W_{\x}^{i+1})^2-2(W_{\x}^i)^2+(W_{\x}^{i-1})^2 \Big).
\ee
We now prove a discrete monotonicity result:
\begin{lem}
\label{lem:monotone t}
Assume that $C\geq \partial_\theta X(0, \te)\geq c$
and $C \geq W_{\x}^i(0)\geq c$ for all $i$ and $\te \in (0,1)$. Then
$C \geq W_{\x}^i(t),W_{X}^i(t) \geq c$ for all $i=1,\ldots,N$, and all $t \geq 0$.
\end{lem}
\begin{proof}
The inequality for $W_X^i$ follows from the fact that the bound
 $C \geq \partial_\theta X(0)\geq c$ is propagated in time (see Proposition \ref{prop:Lagr}
 and recall that here $\lambda=1$).
 
 To prove that $W_{\x}^i(t)\ge c>0$, it suffices to prove that, for any $\e>0$ small,
\be\label{sottosolW}
W^i_{\x}(t)\ge c-\e(2-e^{-t}):=f(t)\quad \forall i,\, \forall\,t \geq 0
\ee
(the bound $W_{\x}^i(t) \leq C$ being obtained in a  is completely analogous manner).
Notice that, with this choice, $f(0)<\min_i W^i_{\x}(0)$.
Suppose by contradiction that 
$$
\min_i W^i_{\x}(t) \not\ge f(t) \ \mbox{in}\ \br^+
$$
Then there exist a first time $t_0$ such that $W^{i_0}_X(t_0)= f(t_0)\ge0$ for some $i_0$, i.e., $f(t) < W^i_{\x}(t)$ for all $t \in [0,t_0)$ and all $i=1,\ldots,n,$ and $f(t)$ touches $W^{i}_{\x}(t)$ from below at $(i_0, t_0).$
From the equation (\ref{eq.W}) and the condition (\ref{sottosolW}) we get a contradiction: indeed,
since $t_0$ is the first contact time we get
\be\nn
\dot{W}^{i_0}_{\x}(t_0)\le\dot{f}(t_0)=-\e e^{-t_0}<0,
\ee
while since $\bigl(W^{i_0+1}_{\x}(t_0)\bigr)^2, \big(W^{i_0-1}_{\x}(t_0)\big)^2 \ge f(t_0)^2=\big(W^{i_0}_{\x}(t_0)\big)^2$
(here we used that $f(t) \geq 0$ provided $\e$ is sufficiently small to deduce that
$W^i \geq f$ implies $(W^i)^2\geq f^2$)
\be\nn
\dot{W}^{i_0}_{\x}(t_0)=\frac{N^2}{4}\Big(\big(W_{\x}^{i_0+1}(t_0)\big)^2-2\big(W_{\x}^{i_0}(t_0)\big)^2+\big(W_{\x}^{i_0-1}(t_0)\big)^2 \Big)\ge0.
\ee
This proves that $\min_i W^i_{\x}(t) \ge f(t)$ for all $t \geq 0$, and
letting $\e \ra 0$ we have the desired result.
%
\end{proof}

We can now prove our convergence theorem.
\begin{thm}
\label{thm:L2}
Let $\x^i$ be a solution of the ODE \eqref{eq:xi 1},
and let $X^i$ be as in \eqref{eq:Xi 1}.
Assume that $X_0 \in C^{4,\alpha}([0,1])$ and that
 there exist positive constants $c_0,C_0$ such that
$$
\frac{c_0}{N}\leq \x^i(0) - \x^{i-1}(0)\leq \frac{C_0}{N},\quad
c_0  \leq \partial_\theta X_0\leq C_0.
$$
Then there exist  two constants $\bar c,\bar C>0$, depending only on $c_0$,
such that 
$$
\frac{1}{N} \sum_{i=1}^N \(\x^i(t)-X^i(t)\)^2
\leq e^{-\bar c t} \frac{1}{N} \sum_{i=1}^N \(\x^i(0)-X^i(0)\)^2
+ \bar C\biggl(\frac{\hat C}{N^2}\biggr)^2
$$
for all $t \geq 0$, where $\hat C$ is as in Lemma \ref{lem:taylor 1}.
\end{thm}

\begin{proof}
We begin by observing that, because of Lemma \ref{lem:monotone t},
$$
\frac{c_0}{N}\leq \x^i(t) - \x^{i-1}(t)\leq \frac{C_0}{N},\ \  \text{and}\ \ 
\frac{c_0}{N}\leq X^i(t) - X^{i-1}(t)\leq \frac{C_0}{N},
$$
for all $t\ge 0.$
We now estimate the $L^2$ distance between $X^i$ and $\x^i$: recalling Lemma \ref{lem:taylor 1}
we have
\begin{align*}
&\frac{d}{dt}\frac{1}{N} \sum_{i=1}^N|{X}^i-{\x}^i|^2\\
&=\frac{1}{8N}\sum_{i=1}^N N\({X}^i-{\x}^i\)\big[Y_{X}^i-Y_{X}^{i-1} - (Y_{\x}^i-Y_{\x}^{i-1}) \big]+\frac{2}{N} \sum_{i=1}^N ({X}^i-{\x}^i) R^i\\
&=\frac{1}{8N} \sum_{i=1}^N N\({X}^i-{\x}^i\)\big[Y_{X}^i- Y_{\x}^i \big]
-\frac{1}{8N} \sum_{i=1}^N N\({X}^i-{\x}^i\)\big[Y_{X}^{i-1} -Y_{\x}^{i-1} \big]\\
&\qquad+\frac{2}{N} \sum_{i=1}^N ({X}^i-{\x}^i) R^i\\
&=\frac{1}{8N} \sum_{i=0}^{N-1} N\({X}^i-{\x}^i\)\big[Y_{X}^i- Y_{\x}^i \big]
-\frac{1}{8N} \sum_{i=0}^{N-1} N\({X}^{i+1}-{\x}^{i+1}\)\big[Y_{X}^i -Y_{\x}^i \big]\\
&\qquad+\frac{2}{N} \sum_{i=1}^N ({X}^i-{\x}^i) R^i\\
&=-\frac{1}{8N} \sum_{i=0}^{N-1} N\(({X}^{i+1}-{X}^i)-({\x}^{i+1}-{\x}^i)\)\big[Y_{X}^i- Y_{\x}^i \big]\\
&\qquad+\frac{2}{N} \sum_{i=1}^N ({X}^i-{\x}^i) R^i.
\end{align*}
Hence
\begin{align*}
&\frac{d}{dt}\frac{1}{N} \sum_{i=1}^N|{X}^i-{\x}^i|^2\\
&=-\frac{1}{8N} \sum_{i=0}^{N-1}(W^i_{X}-W^i_{\x})\big[(W^i_{X})^2- (W^i_{\x})^2 \big]+\frac{2}{N} \sum_{i=1}^N ({X}^i-{\x}^i) R^i\\
&\leq -\frac{c}{8N} \sum_{i=0}^{N-1}(W^i_{X}-W^i_{\x})^2+\frac{2}{N} \sum_{i=1}^N ({X}^i-{\x}^i) R^i,
\end{align*}
since $W_{\x}^i,W_{X}^i \geq c>0$.
We then apply the following discrete Poincar\'e inequality (we postpone the proof to the end of the Theorem):
\begin{lem}
\label{lem:poincare}
Let $(u^0,\ldots,u^N)\subset \mathbb R^N$ with $u^0=0$.
Set 
 $$
\|u\|_2:= \Big(\frac{1}{N}\sum_{i=0}^{N} (u^i)^2 \Big)^{\frac{1}{2}};
$$
$$
\|u'\|_2:= \Big(\frac{1}{N}\sum_{i=0}^{N-1} N^2(u^{i+1}-u^i)^2 \Big)^{\frac{1}{2}}.
$$
Then $\|u\|_2^2\leq \frac{1}2\|u'\|_2^2$.
\end{lem}

$$
\frac{d}{dt}\frac{1}{N} \sum_{i=1}^N|{X}^i-{\x}^i|^2 \leq -\bar c\frac{1}{N} \sum_{i=1}^N|{X}^i-{\x}^i|^2
+ \frac{2}{N} \sum_{i=1}^N ({X}^i-{\x}^i) R^i.
$$
Using that
$$
({X}^i-{\x}^i) R^i \leq \epsilon ({X}^i-{\x}^i)^2 + \frac{1}{\epsilon}(R^i)^2,
$$
choosing $\epsilon=\bar c/4$ we get
$$
\frac{d}{dt}\frac{1}{N} \sum_{i=1}^N|{X}^i-{\x}^i|^2 \leq -\bar c\frac{1}{2N} \sum_{i=1}^N|{X}^i-{\x}^i|^2
+ \frac{2}{N} \sum_{i=1}^N (R^i)^2.
$$
Recalling that
$$
|R^i(t)| \leq \frac{\hat C}{N^2},
$$
(see Lemma \ref{lem:taylor 1}),
we conclude that
$$
\frac{d}{dt}\frac{1}{N} \sum_{i=1}^N|{X}^i-{\x}^i|^2 \leq -\bar c\frac{1}{2N} \sum_{i=1}^N|{X}^i-{\x}^i|^2
+  \frac{2\,\hat C^2}{N^4}.
$$
By Gronwall Lemma, this implies
\begin{align*}
\frac{1}{N} \sum_{i=1}^N|{X}^i(t)-{\x}^i(t)|^2 &\leq \frac{1}{N} \sum_{i=1}^N|{X}^i(0)-{\x}^i(0)|^2\,e^{-\bar c t/2}\\
&\qquad+  \int_0^te^{-\bar c(t-s)/2}\frac{2\,\hat C^2}{N^4}\,ds.
\end{align*}
In particular, using that the third derivatives of $X(t, \cdot)$ are bounded, we get
\begin{align*}
\frac{1}{N} \sum_{i=1}^N|{X}^i(t)-{\x}^i(t)|^2 \leq \frac{1}{N} \sum_{i=1}^N|{X}^i(0)-{\x}^i(0)|^2\,e^{-\bar c t/2}+\frac{2\hat C^2}{\bar c N^4},
\end{align*}
as desired.
\end{proof}

\begin{proof}[Proof of Lemma \ref{lem:poincare}]
We observe that, since $u^0=0$, 
$$
u^i=\frac{1}{N}\sum_{k=0}^{i-1}N(u^{k+1}-u^k) \quad \mbox{for}\ \ i=0, \ldots, N,
$$
hence
\begin{align*}
\|u\|_2^2&=\frac{1}{N}\sum_{i=0}^N (u^i)^2=\frac{1}{N}\sum_{i=0}^N \Big(\frac{1}{N}\sum_{k=0}^{i-1}N(u^{k+1}-u^k)  \Big)^2\\
&\le \frac{1}{N}\sum_{i=0}^N (i-1)\frac{1}{N^2}\sum_{k=0}^{i-1}N^2(u^{k+1}-u^k)^2\\
&=\frac{(N-1)}{2N} \frac{1}{N}\sum_{k=0}^{N-1}N^2(u^{k+1}-u^k)^2\le \frac{1}{2}\|u'\|_2^2.
\end{align*}
\end{proof}

\subsection{The Eulerian picture}

Let us define $\mu^N_t:=\frac{1}{N} \sum_i \delta_{x^i(t)}$.
We want to estimate the distance in $MK_1$ between $\mu_t^N$ and the Lebesgue measure on $[0,1]$.

\begin{thm}
\label{thm:Euler 1}
Let $\x^i$ be a solution of the ODE \eqref{eq:xi 1},
and let $X^i$ be as in \eqref{eq:Xi 1}.
Assume that $X_0 \in C^{4,\alpha}([0,1])$ and that
 there exist positive constants $c_0,C_0$ such that
$$
\frac{c_0}{N}\leq \x^i(0) - \x^{i-1}(0)\leq \frac{C_0}{N},\quad
c_0  \leq \partial_\theta X_0\leq C_0.
$$
Then there exist  two constants $\bar{\bar c},\bar{\bar C}>0$, depending on $c_0,C_0,\|X_0\|_{C^{4,\alpha}([0,1])}$ only, 
such that 
$$
MK_1(\mu_t^N,d\theta) \leq e^{-\bar{\bar c}t/N^3}+\frac{\bar{\bar C}}{N^2}+\frac{1}{4N} \qquad\forall\,t \geq 0.
$$
In particular
$$
MK_1(\mu_t^N,d\theta) \leq \frac{1}{4N} +\frac{\bar{\bar C}+1}{N^2} \qquad \forall\,t \geq \frac{2N^3\log N}{\bar{\bar c}}.
$$
\end{thm}

\begin{proof}
Take $X^0(\theta)=\theta$, so that $X(t,\theta)=\theta$ for all $t$,
and apply Theorem \ref{thm:L2}:
we know that
\begin{align*}
\frac{1}{N} \sum_{i=1}^N|{X}^i(t)-{\x}^i(t)|^2 \leq \frac{1}{N}
\sum_{i=1}^N|{X}^i(0)-{\x}^i(0)|^2\,e^{-\bar c t/2}+\frac{\bar C\,\hat C^2}{N^4},
\end{align*}
hence, since $0 \leq \x^i(0) \leq 1$, $0 \leq X^i(0) \leq 1$,
$$
\frac{1}{N} \sum_{i=1}^N\biggl|{\x}^i(t)-\frac{i-1/2}{N}\biggr|^2 \leq e^{-\bar ct}+\frac{\bar C \,\hat C^2}{N^4}.
$$
Recalling that 
$$
{\x}^i(t):=x^i(N^3t/8),
$$
we get
$$
\frac{1}{N} \sum_{i=1}^N\biggl|x^i(t)-\frac{i-1/2}{N}\biggr|^2 \leq  e^{-\bar ct/N^3}+\frac{\bar C \,\hat C^2}{N^4}.
$$
To control $MK_1(\mu_t^N,d\theta)$, we consider a $1$-Lipschitz function $\varphi$ and we estimate
\begin{align*}
\int_0^1\varphi \,d\mu_t^N - \int_0^1 \varphi\,d\theta&=
\frac{1}{N} \sum_{i=1}^N\varphi(x^i(t))  -\sum_{i=1}^N\int_{(i-1)/N}^{i/N}\varphi\,d\theta\\
&=  \frac{1}{N} \sum_{i=1}^N\biggl[\varphi(x^i(t)) -\varphi\biggl(\frac{i-1/2}{N}\biggr)\biggr]\\
&+\sum_{i=1}^N\int_{(i-1)/N}^{i/N}\biggl[\varphi\biggl(\frac{i-1/2}{N}\biggr) -\varphi(\theta)\biggr]\,d\theta\\
&\leq \frac{1}{N} \sum_{i=1}^N\biggl|x^i(t)-\frac{i-1/2}{N}\biggr|\\
&+\sum_{i=1}^N\int_{(i-1)/N}^{i/N}\biggl|\frac{i-1/2}{N}-\theta\biggr|\,d\theta\\
&\leq \sqrt{\frac{1}{N} \sum_{i=1}^N\biggl|x^i(t)-\frac{i-1/2}{N}\biggr|^2}+\frac{1}{4N}\\
&\leq e^{-\bar ct/(2N^3)}+\frac{\bar C^{1/2} \,\hat C}{N^2}+\frac{1}{4N},
\end{align*}
hence, taking the supremum over all $1$-Lipschitz functions we get
$$
MK_1(\mu_t^N,d\theta) \leq  e^{-\bar ct/(2N^3)}+\frac{\bar C^{1/2} \,\hat C}{N^2}+\frac{1}{4N} ,
$$
which proves the result with $\bar{\bar c}:=\bar c/2$ and $\bar{\bar C}:=\bar C^{1/2} \,\hat C$.
\end{proof}

\section{The case $\rho \not\equiv 1$}
\label{sect:rho neq 1}
We consider the case $r=2$ whit $\rho$ a periodic
function of class $C^{3,\alpha}$, where
$$
\|\rho\|_{C^{3,\alpha}}:= \|\rho\|_{C^3}+ \underset{x\neq y}{\sup}\frac{|\rho'''(x)-\rho'''(y)|}{|x-y|^{\alpha}}.
$$
We recall that
$$
\mc{F}[X]= \frac{1}{12}\int_0^1 \r(X(\te)(\pt_\te X(\te))^{3}d\te,
$$
and that the gradient flow PDE for $\mc{F}$ for the $L^2$-metric is given in
\eqref{gradient flow}.

\subsection{Convergence of the gradient flows}

We recall the formula for the gradient of $F_{N,2}$
given in \eqref{eq:gradient FN2}.

\begin{lem}
\label{lem:taylor}
Let $X(t,\theta)$ be a solution of \eqref{gradient flow}-\eqref{eq:boundary}  starting from an initial datum $X_0 \in C^{4,\alpha}([0,1])$ for some $\alpha>0$ with $\partial_\theta X_0 \geq c_0>0$,
and assume that $0<\lambda \leq \rho \leq 1/\lambda$. Let ${X}^i$ be the discrete values of the exact solution at the points $\left(\frac{i-1/2}{N}, t\right)$ as in \eqref{eq:Xi 1}.
Then
$$
\pt_t {X}^i-N^3\frac{\pt F_{N,2}}{\pt{x^i}}(X^{1}, \ldots, X^{N})=R^i
$$
with
\be
\label{eq:Ri}
|R^i(t)| \leq \frac{\hat C}{N^2} \qquad \forall\,t \geq 0,\,\forall\,i=1,\ldots,N,
\ee
where $\hat C$ depends only on $c_0$, $\lambda$,
$\|\rho\|_{C^{3,\alpha}([0,1])}$, and $\|X_0\|_{C^{4,\alpha}([0,1])}$.
\end{lem}

\begin{proof}
As we showed in Proposition \ref{prop:Lagr}, under our assumptions $\partial_\theta X(t)\geq c>0$ for all $t$ and the solution $X(t)$ remains of class $C^4$ for all times, with
$$
\|X(t)\|_{C^4}\leq C\qquad \forall\,t \geq 0.
$$
A Taylor expansion yields
$$
X^{i+1}=X^i+\frac{1}{N}\pt_\te X^i+\frac{1}{2N^2}\pt_{\te\te} X^i+
\frac{1}{6N^3}\pt_{\te\te\te}X^i +
O\biggl(\frac{\|X(t)\|_{C^4}}{N^4}\biggr);
$$
$$
X^{i-1}=X^i-\frac{1}{N}\pt_\te X^i+\frac{1}{2N^2}\pt_{\te\te} X^i-\frac{1}{6N^3}\pt_{\te\te\te}X^i +
O\biggl(\frac{\|X(t)\|_{C^4}}{N^4}\biggr);
$$
$$
\rho(y)=\r(X^i)+\rho'(X^i) \,(y-X^i)+\frac{\rho''(X^i)}{2}\,(y-X^i)^2 +O(|y-X^i|^3),
$$
where as before we adopt the convention $X^0:=-X^1$ and $X^{N+1}:=2-X^N$.
In addition, we set
$$
\rho(y):=\rho(-y) \quad \text{for }y \in [X^0,0],\qquad
\rho(y):=\rho(2-y) \quad \text{for }y \in [1,X^{N+1}].
$$
Then
\begin{align*}
-\frac{\pt F_{N,2}}{\pt{x^i}}(X^{1}, \ldots, X^{N})= 2\int_{\frac{X^i+X^{i-1}}{2}}^{\frac{X^i+X^{i+1}}{2}} (y-X^i)\biggl[\r(X^i)+\rho'(X^i) \,(y-X^i)&\\
+\frac{\rho''(X^i)}{2}\,(y-X^i)^2 +O(|y-X^i|^3) \biggr]dy&\\
= 2\int_{\frac{X^i+X^{i-1}}{2}}^{\frac{X^i+X^{i+1}}{2}} (y-X^i)\r(X^i)dy+ 2\int_{\frac{X^i+X^{i-1}}{2}}^{\frac{X^i+X^{i+1}}{2}} (y-X^i)\biggl[\rho'(X^i) \,(y-X^i)&\\
+\frac{\rho''(X^i)}{2}\,(y-X^i)^2 +O(|y-X^i|^3)\biggr]dy&
\end{align*}
so that
\begin{align*}
-\frac{\pt F_{N,2}}{\pt{x^i}}(X^{1}, \ldots, X^{N})=\frac{\rho(X^i)}{4} \biggl[(X^{i+1}-X^i)^2- (X^{i}-X^{i-1})^2\biggr]&\\
+2\rho'(X^i) \frac1{24} \biggl[(X^{i+1}-X^i)^3- (X^{i}-X^{i-1})^3\biggr]&\\
+\rho''(X^i) \frac1{64} \biggl[(X^{i+1}-X^i)^4- (X^{i}-X^{i-1})^4\biggr]&\!+O(1/N^5).
\end{align*}
Therefore
\begin{align*}
-\frac{\pt F_{N,2}}{\pt{x^i}}(X^{1}, \ldots, X^{N})
&=\frac{\rho(X^i)}{4} \biggl[(X^{i+1}-X^i)^2- (X^{i}-X^{i-1})^2\biggr]\\
&+\rho'(X^i) \frac1{12} \biggl[(X^{i+1}-X^i)^3- (X^{i}-X^{i-1})^3\biggr]\\
&+\rho''(X^i) \frac1{64} \biggl[(X^{i+1}-X^i)^4- (X^{i}-X^{i-1})^4\biggr]\\
&+O(1/N^5).
\end{align*}
We now use the Taylor expansion for $X$ to see that
$$
(X^{i+1}-X^i)^2- (X^{i}-X^{i-1})^2
=\frac{2 \pt_\te X^i \,\pt_{\te\te} X^i}{N^3} +O(1/N^5),
$$
$$
(X^{i+1}-X^i)^3- (X^{i}-X^{i-1})^3=\frac{2(\pt_\te X^i)^3}{N^3} +O(1/N^5),
$$
$$
(X^{i+1}-X^i)^4- (X^{i}-X^{i-1})^4=O(1/N^5),
$$
thus
\begin{multline*}
-\frac{\pt F_{N,2}}{\pt{x^i}}(X^{1}, \ldots, X^{N})\\
=\frac1{2N^3} \rho(X^i)\pt_\te X^i \,\pt_{\te\te} X^i+
\frac{1}{6N^3}\rho'(X^i)(\pt_\te X^i)^3 +O(1/N^5)=O(1/N^5).
\end{multline*}
\end{proof}

\subsubsection{The $L^{\infty}$ stability estimate}
Let $X$ be a smooth solution of the continuous gradient flow
\be
\label{eq:X}
\pt _tX=\frac12 \r(X)\pt_\te X \,\pt_{\te\te} X+\frac16 \rho'(X)(\pt_\te X)^3
\ee
and define
\be
\label{eq:Xi}
X^i(t):=X\biggl(t,\frac{i-1/2}N\biggr).
\ee
Recall that, according to Lemma \ref{lem:taylor}, $X^i$ solves the following ODE:
\be
\label{eq:ODE Xi}
\dot{X}^i=2N^3 \int_{\frac{X^i+X^{i-1}}{2}}^{\frac{X^i+X^{i+1}}{2}} (z-X^i)\r(z)dz+ R^i
\ee
where $R^i$ satisfies \eqref{eq:Ri} and we are using the conventions  $X^0:=-X^1$, $X^{N+1}:=2-X^N$,
and
$$
\rho(y):=\rho(-y) \quad \text{for }y \in [X^0,0],\qquad
\rho(y):=\rho(2-y) \quad \text{for }y \in [1,X^{N+1}].
$$

We also consider the rescaled discrete solution  $\bigl(\x^i(t) \bigr)_{1\le i\le N}$
\be
\label{eq:xi}
\dot{\x}^i=2N^3 \int_{\frac{\x^i+\x^{i-1}}{2}}^{\frac{\x^i+\x^{i+1}}{2}} (z-\x^i)\r(z)dz.
\ee
In the following lemma we prove that, over a time scale $\tau >0$, $X^i$ gets at most $\eta\tau $ apart
from the exact solution of the ODE, where $\eta$ depends both on $\frac{\hat{C}}{N^2}$ and on the initial distance
between the two solutions.

\begin{lem}
\label{lem:Linfty}
Let $\x^i$ be a solution of the ODE \eqref{eq:xi},
and let $X^i$ be as in \eqref{eq:Xi}.
Set
$$
A_t:=\max_{i=1,\ldots,N} |\x^i({t})-X^i({t})|.
$$

There exists a time $T>0$, depending only on $\sup_{t\geq 0}\|X(t)\|_{C^2}$
and $\|\rho'\|_\infty+\|\rho''\|_\infty$, such that,
for any $t^*\geq 0$,
$$
A_{t^*+\tau }\leq A_{t^*}+ \eta \,\tau \qquad \forall\,\tau \in [0,T]
$$
with $\eta:=\frac{3\hat{C}}{N^2}+\frac{A_{t^*}}{T}$, where $\hat C$ is as in \eqref{eq:Ri}.
\end{lem}

\begin{proof}
Let us define
$$
A_{t}^\pm:=\max_{i=1,\ldots,N}\bigl( \pm [\x^i({t})-X^i({t})]\bigr)_+.
$$
Notice that $A_t=\max\{A_t^+,A_t^-\}$, and to prove the result it is enough to prove the following
stronger statement:
\be
\label{eq:stability}
A_{t^*+\tau}^+\leq A_{t^*}^+ + \eta \,\tau\qquad \forall\,\tau \in [0,T],
\ee
$$
A_{t^*+\tau}^-\leq A_{t^*}^- - \eta \,\tau\qquad \forall\,\tau \in [0,T].
$$
Since the arguments for $A^+$ and $A^-$ are completely analogous , we prove only \eqref{eq:stability}.
Also, without loss of generality we can assume $t^*=0.$ 
By definition of $A^+_0,$ at time $0$ the solutions are ordered
$$
X^i(0)\le \x^i(0)+A^+_0.
$$
Let us define $Y^i(t):=X^i(t)-A^+_0-\eta t$ and assume that there exist $t_0\in \br^+$ defined as $t_0:=\underset{t\in \br^+}{\inf} \bigl\{ Y^i(t)=\x^i(t)\bigr\}.$
Then,
\be
\label{eq:eta}
\dot{\x}^i(t_0)\le \dot{Y}^i(t_0)\le  \dot{X}^i(t_0)-\eta.
\ee
Observing that  $\x^{i+1}(t_0) \ge Y^{i+1}(t_0)$ and $\x^{i-1}(t_0) \ge Y^{i-1}(t_0),$ 
\begin{align*}
\dot{\x}^i(t_0)=&2N^3 \int_{\frac{\x^i+\x^{i-1}}{2}}^{\frac{\x^i+\x^{i+1}}{2}} (z-\x^i)\r(z)dz\\
&\ge 2N^3 \int_{\frac{Y^i+Y^{i-1}}{2}}^{\frac{Y^i+Y^{i+1}}{2}} (z-Y^i)\r(z)dz.
\end{align*}
Performing a change of variable $\omega=z+A^+_{0}+\eta t_0,$ we have
 \begin{align*}
\dot{\x}^i(t_0)&\ge 2N^3 \int_{\frac{Y^i+Y^{i-1}}{2}}^{\frac{Y^i+Y^{i+1}}{2}} (z-Y^i)\r(z)dz\\
&=2N^3 \int_{\frac{X^i+X^{i-1}}{2}}^{\frac{X^i+X^{i+1}}{2}} (\omega-X^i)\r(\omega-A^+_{0}-\eta t_0)d\omega.
 \end{align*}
By the fundamental theorem of calculus 
\begin{align*}
\r(\omega-A^+_{0}-\eta t_0)&=\r(\omega)-(A^+_{0}+\eta t_0)\(\int_0^1 \r'(\omega+s(A^+_{0}+\eta t_0)) ds\)\\
&:=\r(\omega)-a(\omega),
\end{align*}
so
 \begin{align*}
\dot{\x}^i(t_0)&\ge 2N^3 \int_{\frac{X^i+X^{i-1}}{2}}^{\frac{X^i+X^{i+1}}{2}} (\omega-X^i)\r(\omega)d\omega\\
&-2N^3(A^+_{0}+\eta t_0) \int_{\frac{X^i+X^{i-1}}{2}}^{\frac{X^i+X^{i+1}}{2}} (\omega-X^i)a(\omega)d\omega.
\end{align*}
If we recall that $X^i$ solves the ODE (\ref{eq:ODE Xi}) we have
\begin{align*}
\dot{\x}^i(t_0)&\ge \dot{X}^i-R^i-2N^3(A^+_{0}+\eta t_0) \int_{\frac{X^i+X^{i-1}}{2}}^{\frac{X^i+X^{i+1}}{2}} (\omega-X^i)a(\omega)d\omega\\
&= \dot{X}^i-R^i-2N^3(A^+_{0}+\eta t_0) \int_{\frac{X^i+X^{i-1}}{2}}^{\frac{X^i+X^{i+1}}{2}} (\omega-X^i)\(a(\omega)-a(X^i)\)d\omega\\
& +2N^3(A^+_{0}+\eta t_0) \int_{\frac{X^i+X^{i-1}}{2}}^{\frac{X^i+X^{i+1}}{2}} (\omega-X^i)a(X^i)d\omega\\
&:=\dot{X}^i-R^i- T_1+T_2.
\end{align*}
For $T_1$ we observe that, since $|X^{i+1}-X^i|\leq C/N$ for all $i$,
$$
|T_1|\leq CN^3(A^+_{0}+\eta t_0) \|a'\|_\infty \int_{\frac{X^i+X^{i-1}}{2}}^{\frac{X^i+X^{i+1}}{2}}|\omega-X^i|^2d\omega
\leq C(A^+_{0}+\eta t_0) \|\rho''\|_\infty.
$$
For $T_2$ we use the Taylor expansion for $X$:
$$
X^{i+1}=X^i+\frac{\pt_\te X^i}{N}+O\(\frac{1}{N^2}\);
$$
$$
X^{i-1}=X^i-\frac{\pt_\te X^i}{N}+O\(\frac{1}{N^2}\).
$$
Thus,
\begin{align*}
T_2&\le CN^3(A^+_{0}+\eta t_0) \|\rho'\|_\infty \int_{\frac{X^i+X^{i-1}}{2}}^{\frac{X^i+X^{i+1}}{2}} (\omega-X^i)d\omega\\
&=CN^3(A^+_{0}+\eta t_0) \|\rho'\|_\infty\Biggl[-\frac12\(-\frac{\pt_\te X^i}{N}+O\(\frac{1}{N^2}\)\)^2\\
&\quad\qquad\qquad\qquad\qquad\qquad+\frac12\(\frac{\pt_\te X^i}{N}+O\(\frac{1}{N^2}\)\)^2 \Biggr]\\
&\le C(A^+_{0}+\eta t_0) \|\rho'\|_\infty.
\end{align*}
Then
$$
\dot{\x}^i(t_0)\ge \dot{X}^i-|R^i| - C(A^+_{0}+\eta t_0) \(\|\rho''\|_\infty +\|\rho'\|_\infty\),
$$
that combined with \eqref{eq:eta} and \eqref{eq:Ri} gives
\begin{align*}
\eta &\leq C(A^+_{0}+\eta t_0) \(\|\rho''\|_\infty +\|\rho'\|_\infty\)+ |R^i|\\
&\leq C(A^+_{0}+\eta t_0) \(\|\rho''\|_\infty +\|\rho'\|_\infty\) +\frac{\hat C}{N^2}.
\end{align*}
We now show that there exists a time $T>0$, depending on $\sup_{t\geq 0}\|X(t)\|_{C^2}$ 
and $\|\rho'\|_\infty+\|\rho''\|_\infty$ only, such that $t_0>T$. This will prove that \eqref{eq:stability} holds on $[0,T]$.

Assume by contradiction that $t_0\leq T$. Then the above estimate gives
$$
\eta\leq C\biggl(\frac{A^+_{0}}{T}+\eta\biggr)T \(\|\rho''\|_\infty +\|\rho'\|_\infty\) +\frac{\hat C}{N^2}.
$$
Choosing $T$ sufficiently small so that
$$
CT\(\|\rho''\|_\infty +\|\rho'\|_\infty\) \leq \frac{1}{2}
$$
we get
$$
\eta \leq \frac{1}{2}\biggl(\frac{A^+_{0}}{T}+\eta\biggr)  +\frac{\hat C}{N^2},
$$
or equivalently
$$
\eta \leq \frac{A^+_{0}}{T}+ 2\frac{\hat C}{N^2}.
$$
This contradicts the definition of $\eta$ and proves the result.
\end{proof}

\subsubsection{The $L^2$ stability estimate}

\begin{lem}
\label{lem:L2}
Let $\x^i$ be a solution of the ODE \eqref{eq:xi},
and let $X^i$ be as in \eqref{eq:Xi}.
Let $0 \leq T_1\leq T_2\leq \infty$, and
assume that there exist two positive constants $c_0,C_0$ such that
$$
\frac{c_0}{N}\leq \x^i(t) - \x^{i-1}(t)\leq \frac{C_0}{N},\quad
\frac{c_0}{N}\leq X^i(t) - X^{i-1}(t)\leq \frac{C_0}{N},\qquad \forall\,t \in [T_1,T_2].
$$
Then, there exists $\e_0=\e_0(c_0, C_0)>0$ such that, if $\|\r'\|_{L^\infty}+ \|\r''\|_{L^\infty}\leq \e_0$
then one can find two constants $\bar c,\bar C>0$, depending only on $c_0$,
such that 
$$
\frac{1}{N} \sum_{i=1}^N \(\x^i(t)-X^i(t)\)^2
\leq e^{-\bar c (t-T_1)} \frac{1}{N} \sum_{i=1}^N \(\x^i(T_1)-X^i(T_1)\)^2
+ \bar C\biggl(\frac{\hat C}{N^2}\biggr)^2
$$
for all $t \in [T_1,T_2]$.
\end{lem}

\begin{proof} We compute
\begin{align*}
&\frac{d}{dt}\frac{1}{N} \sum_{i=1}^N \(\x^i-X^i\)^2=\\
&4N^2 \sum_{i=1}^N \(\x^i-X^i\)\[ \int_{\frac{\x^i+\x^{i-1}}{2}}^{\frac{\x^i+\x^{i+1}}{2}} (z-\x^i)\r(z)dz-\int_{\frac{X^i+X^{i-1}}{2}}^{\frac{X^i+X^{i+1}}{2}} (z-X^i)\r(z)dz\]\\
&+\frac{2}{N} \sum_{i=1}^N \(\x^i-X^i\)R^i\\
&=4N^2 \sum_{i=1}^N \(\x^i-X^i\)\Bigg[\int_{\frac{\x^i+\x^{i-1}}{2}}^{\x^i} (z-\x^i)\r(z)dz\\
&+\int_{\x^i}^{\frac{\x^i+\x^{i+1}}{2}} (z-\x^i)\r(z)dz-\int_{\frac{X^i+X^{i-1}}{2}}^{X^i} (z-X^i)\r(z)dz\\
&-\int_{X^i}^{\frac{X^i+X^{i+1}}{2}} (z-X^i)\r(z)dz\Bigg]+\frac{2}{N} \sum_{i=1}^N \(\x^i-X^i\)R^i\\
&:=4N^2 \sum_{i=1}^N \(\x^i-X^i\)\bigg[A_{\x^i}+B_{\x^i}-A_{X^i}-B_{X^i} \bigg]+\frac{2}{N} \sum_{i=1}^N \(\x^i-X^i\)R^i.\\
\end{align*}
For $A_{\x^i}$ and $B_{\x^i}$  we have
\begin{align*}
&A_{\x^i}=\int_{\frac{\x^i+\x^{i-1}}{2}}^{\x^i} (z-\x^i)\r(\x^{i-1})dz+\int_{\frac{\x^i+\x^{i-1}}{2}}^{\x^i} (z-\x^i)\big(\r(z)-\r(\x^{i-1})\big)dz\\
&=-\frac{\r(\x^{i-1})}{8}(\x^i-\x^{i-1})^2+\int_{\frac{\x^i+\x^{i-1}}{2}}^{\x^i} (z-\x^i)\big(\r(z)-\r(\x^{i-1})\big)dz\\
&:=D^{i-1}_{\x}+E^{i,2}_{\x}.
\end{align*}
\begin{align*}
&B_{\x^i}=\int_{\x^i}^{\frac{\x^i+\x^{i+1}}{2}} (z-\x^i)\r(\x^{i})dz+\int^{\frac{\x^i+\x^{i+1}}{2}}_{\x^i} (z-\x^i)\big(\r(z)-\r(\x^{i})\big)dz\\
&=\frac{\r(\x^{i})}{8}(\x^{i+1}-\x^{i})^2+\int^{\frac{\x^i+\x^{i+1}}{2}}_{\x^i} (z-\x^i)\big(\r(z)-\r(\x^{i})\big)dz\\
&:=D^{i}_{\x}+E^{i,1}_{\x}.
\end{align*}
Analogously we can set $A_{X^i}:=D^{i-1}_{X}+E^{i,2}_{X}$ and $B_{X^i}:=D^{i}_{X}+E^{i,1}_{X}.$ In this way we have
\begin{align*}
&\frac{d}{dt}\frac{1}{N} \sum_{i=1}^N \(\x^i-X^i\)^2\\
&=4N^2 \sum_{i=1}^N \(\x^i-X^i\)\bigg[A_{\x^i}+B_{\x^i}-A_{X^i}-B_{X^i} \bigg]+\frac{2}{N} \sum_{i=1}^N \(\x^i-X^i\)R^i\\
&=4N^2 \sum_{i=1}^N \(\x^i-X^i\)\bigg[ D^{i}_{\x}-D^{i-1}_{\x}-D^{i}_{X}+D^{i}_{X}\bigg] \\
&+ 4N^2 \sum_{i=1}^N \(\x^i-X^i\)\bigg[ E^{i,1}_{\x}-E^{i,1}_{X}+E^{i,2}_{\x}-E^{i,2}_{X}\bigg]+\frac{2}{N} \sum_{i=1}^N \(\x^i-X^i\)R^i\\
&=T_1 +T_2+\frac{2}{N} \sum_{i=1}^N \(\x^i-X^i\)R^i.
\end{align*}
Let us estimate $T_1$ and $T_2$ separately.
First,
\begin{align*}
T_1&=4N^2 \sum_{i=1}^N \(\x^i-X^i\)\bigg[ D^{i}_{\x}-D^{i-1}_{\x}-D^{i}_{X}+D^{i}_{X}\bigg]\\
&=4N^2\Bigg( \sum_{i=1}^N \(\x^i-X^i\) \(D^{i}_{\x}-D^{i}_{X} \) - \sum_{i=1}^N \(\x^i-X^i\) \(D^{i-1}_{\x}-D^{i-1}_{X} \)\Bigg)\\
\end{align*}
Using the discrete version of the integration by parts we obtain
\begin{align*}
T_1&=4N^2\Bigg(\sum_{i=1}^N \(\x^i-X^i\) \(D^{i}_{\x}-D^{i}_{X} \) - \sum_{i=1}^N \(\x^i-X^i\)\(D^{i-1}_{\x}-D^{i-1}_{X} \)\Bigg)\\
&=4N^2\Bigg( \sum_{i=1}^N \(\x^i-X^i\) \(D^{i}_{\x}-D^{i}_{X} \) - \sum_{i=1}^N \(\x^{i+1}-X^{i+1}\)\(D^{i}_{\x}-D^{i}_{X} \)\Bigg)\\
&=4N^2\Bigg(\sum_{i=1}^N\bigg(\(\x^i-\x^{i+1}\)-\(X^i-X^{i+1}\)\bigg)\(D^{i}_{\x}-D^{i}_{X}\) \Bigg)
\end{align*}
Recalling the definitions of $D^{i}_{\x}$ and $D^{i}_{X}$ we have
\begin{align*}
T_1&=-\frac{N^2}{4}\Bigg(\sum_{i=1}^N\bigg(\(\x^{i+1}-\x^{i}\)-\(X^{i+1}-X^{i}\)\bigg)\\
&\qquad\qquad\qquad\times\bigg(\r(\x^{i})(\x^{i+1}-\x^{i})^2-\r(X^{i})(X^{i+1}-X^{i})^2\bigg) \Bigg)\\
&=-\frac{N^2}{4}\Bigg(\sum_{i=1}^N\bigg[\(\x^{i+1}-\x^{i}\)-\(X^{i+1}-X^{i}\)\bigg]\\
&\qquad\qquad\qquad\times\bigg[\r(\x^{i})\bigg( (\x^{i+1}-\x^{i})^2-(X^{i+1}-X^{i})^2\bigg)\bigg] \Bigg)\\
&+\frac{N^2}{4}\Bigg(\sum_{i=1}^N\bigg[\(\x^{i+1}-\x^{i}\)-\(X^{i+1}-X^{i}\)\bigg]\\
&\qquad\qquad\qquad\times\(\r(\x^{i})-\r(X^i)\)\(X^{i+1}-X^{i}\)^2\Bigg)\\
&=:T_{1,1}+T_{1,2}.
\end{align*}
Notice that, since $\|\rho'\|_\infty \leq \e_0$ and $\|\rho\|_{L^1}=1$, we have
$\rho\geq 1/2$ provided $\e_0$ is small enough.
Hence, recalling that $\x^{i+1}-\x^{i}\ge\frac {c_0}N$ and $X^{i+1}-X^{i}\ge\frac{c_0}N$, we can estimate the first term
\begin{align*}
T_{1,1}&\le -\frac{N^2}{8}\sum_{i=1}^N\bigg[\(\x^{i+1}-\x^{i}\)-\(X^{i+1}-X^{i}\)\bigg]^2\\
&\qquad\qquad\qquad\times\bigg( (\x^{i+1}-\x^{i})+(X^{i+1}-X^{i})\bigg)\\
&\le -\frac{c_0}{4} N\sum_{i=1}^N\bigg[\(\x^{i+1}-\x^{i}\)-\(X^{i+1}-X^{i}\)\bigg]^2\\
&=-\frac{c_0}{4N}\sum_{i=1}^N\bigg[N\(\x^{i+1}-\x^{i}\)-N\(X^{i+1}-X^{i}\)\bigg]^2.
\end{align*}
Hence, recalling that $X^{i+1}-X^{i}\le \frac{C_0}{N}$,
\begin{align*}
|T_{1,2}|&\le \|\r'\|_{L^\infty}\frac{N^2}{4}\sum_{i=1}^N\bigg|\(\x^{i+1}-\x^{i}\)-\(X^{i+1}-X^{i}\)\bigg|\\
&\qquad\qquad\qquad\qquad\times\( \x^i-X^i\)\(X^{i+1}-X^{i}\)^2\\
&\le\frac{C_0^2}{N} \|\r'\|_{L^\infty}\sum_{i=1}^N\bigg|N\(\x^{i+1}-\x^{i}\)-N\(X^{i+1}-X^{i}\)\bigg| \big| \x^i-X^i\big|.
\end{align*}
Using the inequality $ab\le a^2+b^2$ we get
\begin{align*}
|T_{1,2}|&\le \frac{C_0^2}{N} \|\r'\|_{L^\infty}\sum_{i=1}^N\bigg[N\(\x^{i+1}-\x^{i}\)-N\(X^{i+1}-X^{i}\)\bigg]^2\\
&+ \frac{C_0^2}{N} \|\r'\|_{L^\infty}\sum_{i=1}^N \( \x^i-X^i\)^2.
\end{align*}
Let us now consider $T_2.$ 
\begin{align*}
 &T_2=4N^2 \sum_{i=1}^N \(\x^i-X^i\)\bigg[ E^{i,1}_{\x}-E^{i,1}_{X}+E^{i,2}_{\x}-E^{i,2}_{X}\bigg]\\
 &=4N^2 \sum_{i=1}^N \(\x^i-X^i\)\bigg[ E^{i,1}_{\x}-E^{i,1}_{X}\bigg]+4N^2 \sum_{i=1}^N \(\x^i-X^i\)\bigg[ E^{i,2}_{\x}-E^{i,2}_{X}\bigg]\\
 &:=T_{2,1}+T_{2,2}.
\end{align*}
Let us first focus on the differences $E^{i,1}_{\x}-E^{i,1}_{X}$ and $E^{i,2}_{\x}-E^{i,2}_{X}.$ Keeping in mind the definitions of $E^{i,1}_{\x}$ and $E^{i,1}_{X}$ we have
\begin{align*}
&E^{i,1}_{\x}-E^{i,1}_{X}=\\
&\int_{\x^i}^{\frac{\x^i+\x^{i+1}}{2}}(z-\x^i)\(\r(z)-\r(\x^i)\)dz-\int_{X^i}^{\frac{X^i+X^{i+1}}{2}}(z-X^i)\(\r(z)-\r(X^i)\)dz.
\end{align*}
Performing the change of variable $\omega=z-\x^i,$ $\omega=z-X^i$ respectively, we get
\begin{align*}
E^{i,1}_{\x}-E^{i,1}_{X}&=\int_{0}^{\frac{\x^{i+1}-\x^{i}}{2}}\omega\(\r(\omega+\x^i)-\r(\x^i)\)d\omega\\
&-\int_{0}^{\frac{X^{i+1}-X^{i}}{2}}\omega\(\r(\omega+X^i)-\r(X^i)\)d\omega.
\end{align*}
Adding and subtracting 
$$
\int_{0}^{\frac{\x^{i+1}-\x^{i}}{2}}\omega\(\r(\omega+X^i)-\r(X^i)\)d\omega
$$
we have
\begin{align*}
E^{i,1}_{\x}-E^{i,1}_{X}&=\int_{0}^{\frac{\x^{i+1}-\x^{i}}{2}}\omega\bigg[ \r(\omega+\x^i)-\r(\x^i)-\r(\omega+X^i)+\r(X^i)\bigg]d\omega\\
&-\int_{\frac{\x^{i+1}-\x^{i}}{2}}^{\frac{X^{i+1}-X^{i}}{2}}\omega\(\r(\omega+X^i)-\r(X^i)\)d\omega.\\
\end{align*}
By the fundamental theorem of calculus and recalling that $(\x^{i+1}-\x^{i})\le\frac {C_0}N,$ $(X^{i+1}-X^{i})\le\frac{C_0}N$ we obtain the following estimate 
\begin{align*}
|E^{i,1}_{\x}-E^{i,1}_{X}|&=\bigg|\int_{0}^{\frac{\x^{i+1}-\x^{i}}{2}}\omega^2\[\int_0^1\r'(\x^i+s\omega)ds-\int_0^1\r'(X^i+s\omega)ds \]d\omega\\
&-\int_{\frac{\x^{i+1}-\x^{i}}{2}}^{\frac{X^{i+1}-X^{i}}{2}}\omega\(\r(\omega+X^i)-\r(X^i)\)d\omega\bigg|\\
&\le\frac{C_0}{N^3} \|\r''\|_{L^\infty}|\x^i-X^i|+ \|\r'\|_{L^\infty}\Bigg|\int_{\frac{\x^{i+1}-\x^{i}}{2}}^{\frac{X^{i+1}-X^{i}}{2}}\omega^2d\omega\Bigg|\\
&=\frac{C_0}{N^3} \|\r''\|_{L^\infty}|\x^i-X^i|\\
&+ \frac{\|\r'\|_{L^\infty}}{8}\bigg|\(X^{i+1}-X^i\)^3-\(\x^{i+1}-\x^i\)^3\bigg|.
\end{align*}
Thus, 
\begin{align*}
|T_{2,1}|&=4N^2 \sum_{i=1}^N \big| \x^i-X^i\big|\big| E^{i,1}_{\x}-E^{i,1}_{X}\big|\\
&\le \frac{C}{N} \|\r''\|_{L^\infty} \sum_{i=1}^N \(\x^i-X^i\)^2\\
&+\frac{N^2}{2}\|\r'\|_{L^\infty}\sum_{i=1}^N |\x^i-X^i|\Big|\(X^{i+1}-X^i\)^3-\(\x^{i+1}-\x^i\)^3\Big|.\\
\end{align*}
Recalling that $0\leq (\x^{i+1}-\x^{i})\le\frac {C_0}N$ and $0 \leq (X^{i+1}-X^{i})\le\frac{C_0}N$
we see that
$$
\Big|\(X^{i+1}-X^i\)^3-\(\x^{i+1}-\x^i\)^3\Big|\leq \frac{C}{N^2}
\Big|\(X^{i+1}-X^i\)-\(\x^{i+1}-\x^i\)\Big|,
$$
therefore
\begin{align*}
|T_{2,1}|&\le \frac{C}{N} \|\r''\|_{L^\infty} \sum_{i=1}^N \(\x^i-X^i\)^2\\
&+\frac{C}{N} \|\r'\|_{L^\infty} \sum_{i=1}^N |\x^i-X^i| \big| N\(\x^{i+1}-\x^i\)-N\(X^{i+1}-X^i\)\big|\\
&\le \frac{C}{N} \|\r''\|_{L^\infty} \sum_{i=1}^N \(\x^i-X^i\)^2\\
&+\frac{C}{N} \|\r'\|_{L^\infty} \Bigg[\sum_{i=1}^N |\x^i-X^i|^2\!+\!\sum_{i=1}^N \big| N\(\x^{i+1}-\x^i\)\!-\!N\(X^{i+1}-X^i\)\big|^2\Bigg].
\end{align*}
Let us now estimate $E^{i,2}_{\x}-E^{i,2}_{X}.$ By definition we have
\begin{align*}
E^{i,2}_{\x}-E^{i,2}_{X}=\int^{\x^i}_{\frac{\x^i+\x^{i-1}}{2}}(z-\x^i)\(\r(z)-\r(\x^{i-1})\)dz&\\
-\int^{X^i}_{\frac{X^i+X^{i-1}}{2}}(z-X^i)\(\r(z)-\r(X^{i-1})\)dz&.
\end{align*}
With the substitutions $\omega=z-\x^{i-1}$ and $\omega=z-X^{i-1}$ respectively, we get
\begin{align*}
E^{i,2}_{\x}-E^{i,2}_{X}=\int_{\frac{\x^i-\x^{i-1}}{2}}^{\x^{i}-\x^{i-1}}\(\omega+\x^{i-1}-\x^i\)\(\r(\omega+\x^{i-1})-\r(\x^{i-1})\)d\omega&\\
-\int_{\frac{X^i-X^{i-1}}{2}}^{X^{i}-X^{i-1}}\(\omega+X^{i-1}-X^i\)\(\r(\omega+X^{i-1})-\r(X^{i-1})\)d\omega&.
\end{align*}
Adding and subtracting 
$$
-\int_{\frac{X^i-X^{i-1}}{2}}^{X^{i}-X^{i-1}}\(\omega+X^{i-1}-X^i\)\(\r(\omega+X^{i-1})-\r(X^{i-1})\)d\omega
$$
we get
\begin{align*}
|E^{i,2}_{\x}-E^{i,2}_{X}|\le\Bigg|\int_{\frac{\x^i-\x^{i-1}}{2}}^{\x^{i}-\x^{i-1}}\(\omega+\x^{i-1}-\x^i\)\(\r(\omega+\x^{i-1})-\r(\x^{i-1})\right.\\
\left.-\r(\omega+X^{i-1})+\r(X^{i-1})\)d\omega\Bigg|\\
+ \Bigg|-\int_{\frac{X^i-X^{i-1}}{2}}^{X^{i}-X^{i-1}}\(\omega+X^{i-1}-X^i\)\(\r(\omega+X^{i-1})-\r(X^{i-1})\)d\omega \\
+ \int_{\frac{\x^i-\x^{i-1}}{2}}^{\x^{i}-\x^{i-1}}\(\omega+X^{i-1}-X^i\)\(\r(\omega+X^{i-1})-\r(X^{i-1})\)d\omega \Bigg|.\\
\end{align*}
Arguing as we did for the first term in $E^{i,1}_{\x}-E^{i,1}_{X}$,
the first term in $E^{i,2}_{\x}-E^{i,2}_{X}$
is controlled by
$$
\|\rho''\|_\infty\int_{\frac{\x^i-\x^{i-1}}{2}}^{\x^{i}-\x^{i-1}}\big|\omega+\x^{i-1}-\x^i\big|\Big|X^{i-1}-\x^{i-1}\Big|\omega^2\,d\omega,
$$
and recalling that $|\x^{i-1}-\x^i| \leq C_0/N$, the above term is bounded by
$$
\frac{C}{N^3}\|\rho''\|_\infty\Big|X^{i-1}-\x^{i-1}\Big|.
$$
Concerning the second term in $E^{i,2}_{\x}-E^{i,2}_{X}$,
using that
$$
\biggl|\int_{a/2}^a - \int_{b/2}^b \biggr|\leq 
\biggl|\int_0^a - \int_0^b \biggr|+ 
\biggl|\int_0^{a/2} - \int_0^{b/2} \biggr|=\biggl|\int_a^b\biggr|
+\biggl|\int_{a/2}^{b/2}\biggr|
$$
we get
\begin{align*}
 &\Bigg|\int_{\frac{\x^i-\x^{i-1}}{2}}^{\x^{i}-\x^{i-1}}\(\omega+X^{i-1}-X^i\)\(\r(\omega+X^{i-1})-\r(X^{i-1})\)d\omega\\
&-\int_{\frac{X^i-X^{i-1}}{2}}^{X^{i}-X^{i-1}}\(\omega+X^{i-1}-X^i\)\(\r(\omega+X^{i-1})-\r(X^{i-1})\)d\omega \Bigg|\\
& \le \Bigg|\int_{\frac{\x^i-\x^{i-1}}{2}}^{\frac{X^i-X^{i-1}}{2}}\(\omega+X^{i-1}-X^i\)\(\r(\omega+X^{i-1})-\r(X^{i-1})\)d\omega  \Bigg|\\
&+  \Bigg|\int_{\x^i-\x^{i-1}}^{X^i-X^{i-1}}\(\omega+X^{i-1}-X^i\)\(\r(\omega+X^{i-1})-\r(X^{i-1})\)d\omega  \Bigg|.\\
&\le \|\r'\|_{L^{\infty}}\Bigg[ \Big|\int_{\frac{\x^i-\x^{i-1}}{2}}^{\frac{X^i-X^{i-1}}{2}}\big|\omega+X^{i-1}-X^i\big| \,|\omega|\, d\omega\Big|\\ 
&\qquad\qquad+\Big|\int_{\x^i-\x^{i-1}}^{X^i-X^{i-1}}\big|\omega+X^{i-1}-X^i\big|\,|\omega|\, d\omega  \Big|\Bigg].
\end{align*}
We now notice that in the last term the second integral is bounded by the first integral
hence we can bound it by
\begin{align*}
&2\|\r'\|_{L^{\infty}} \int_{\x^i-\x^{i-1}}^{X^i-X^{i-1}}\omega^2\,d\omega + 
2\|\r'\|_{L^{\infty}} (X^{i}-X^{x-i})\int_{\x^i-\x^{i-1}}^{X^i-X^{i-1}}\omega\,d\omega\\
&\leq C\|\r'\|_{L^{\infty}}\Big| (X^i-X^{i-1})^3-(\x^i-\x^{i-1})^3\Big|\\
&\qquad +C\|\r'\|_{L^{\infty}}(X^{i-1}-X^i)\Big| (X^i-X^{i-1})^2-(\x^i-\x^{i-1})^2\Big|.
\end{align*}
Hence, arguing as for $T_{2,1}$, we obtain
\begin{align*}
|T_{2,2}|\le \frac{C}{N} \|\r''\|_{L^\infty} \sum_{i=1}^N \(\x^i-X^i\)^2+\frac{C}{N} \|\r'\|_{L^\infty} \Bigg[\sum_{i=1}^N |\x^i-X^i|^2&\\
+\sum_{i=1}^N \big| N\(\x^{i+1}-\x^i\)-N\(X^{i+1}-X^i\)\big|^2\Bigg]&.
\end{align*}
Combining all these bounds together, we get
\begin{align*}
&\frac{d}{dt}\frac{1}{N} \sum_{i=1}^N \(\x^i-X^i\)^2\\
&=T_1 +T_2+\frac{2}{N} \sum_{i=1}^N \(\x^i-X^i\)R^i\\
&=T_{1,1}+T_{1,2}+T_{2,1}+T_{2,2}+\frac{2}{N} \sum_{i=1}^N \(\x^i-X^i\)R^i\\
&\leq -\frac{c_0\lambda}{2N}\sum_{i=1}^N\Big[N\(\x^{i+1}-\x^{i}\)-N\(X^{i+1}-X^{i}\)\Big]^2\\
&\qquad+\frac{C}{N}\Big(\|\r'\|_{L^\infty}+ \|\r''\|_{L^\infty}\Bigr)\sum_{i=1}^N \(\x^i-X^i\)^2\\
&\qquad +\frac{C}{N}\Big(\|\r'\|_{L^\infty}+ \|\r''\|_{L^\infty}\Bigr)\sum_{i=1}^N \Big[N\(\x^{i+1}-\x^{i}\)-N\(X^{i+1}-X^{i}\)\Big]^2\\
&\qquad+\frac{2}{N} \sum_{i=1}^N \(\x^i-X^i\)R^i.
\end{align*}
Hence, recalling that $\|\r'\|_{L^\infty}+ \|\r''\|_{L^\infty}\leq \e_0$,
we can choose $\e_0$
small  (the smallness
depending only on $c_0,C_0,\lambda$)
so that $C\Big(\|\r'\|_{L^\infty}+ \|\r''\|_{L^\infty}\Bigr) \leq c_0\lambda/2$ to
obtain
\begin{align*}
\frac{d}{dt}\frac{1}{N} \sum_{i=1}^N \(\x^i-X^i\)^2
&\leq -\frac{c_0}{8N}\sum_{i=1}^N\Big[N\(\x^{i+1}-\x^{i}\)-N\(X^{i+1}-X^{i}\)\Big]^2\\
&\qquad+\frac{C}{N}\e_0\sum_{i=1}^N \(\x^i-X^i\)^2+\frac{2}{N} \sum_{i=1}^N \(\x^i-X^i\)R^i.
\end{align*}
We now use the discrete Poincar\'e inequality (see Lemma \ref{lem:poincare}) to get
$$
\frac12 \sum_{i=1}^N\Big[N\(\x^{i+1}-\x^{i}\)-N\(X^{i+1}-X^{i}\)\Big]^2
\geq \sum_{i=1}^N \(\x^i-X^i\)^2,
$$
so that assuming $\e_0$ small enough we conclude 
\begin{align*}
\frac{d}{dt}\frac{1}{N} \sum_{i=1}^N \(\x^i-X^i\)^2&
\leq -\frac{c_0}{N}\sum_{i=1}^N\(\x^i-X^i\)^2+\frac{C}{N}\e_0\sum_{i=1}^N \(\x^i-X^i\)^2\\
&\qquad+\frac{2}{N} \sum_{i=1}^N \(\x^i-X^i\)R^i\\
&\leq -\frac{2c_0}{3N}\sum_{i=1}^N\(\x^i-X^i\)^2+\frac{2}{N} \sum_{i=1}^N \(\x^i-X^i\)R^i
\end{align*}
Finally,
using the bound
$$
2\(\x^i-X^i\)R^i \leq \epsilon (\x^i-X^i)^2+\frac{1}{\epsilon}|R^i|^2
$$
with $\epsilon:=2c_0/3$, and recalling that $|R^i|\leq \hat C/N^2$ we conclude 
$$
\frac{d}{dt}\frac{1}{N} \sum_{i=1}^N \(\x^i-X^i\)^2
\leq -\frac{c_0}{6N}\sum_{i=1}^N\(\x^i-X^i\)^2
+\frac{3}{2c_0} \biggl(\frac{\hat C}{N^2}\biggr)^2.
$$
Integrating this differential inequality over $[T_1,t]$ with $t \leq T_2$, by Gronwall Lemma we obtain
$$
\frac{1}{N} \sum_{i=1}^N \(\x^i(t)-X^i(t)\)^2
\leq e^{-\bar c (t-T_1)} \frac{1}{N} \sum_{i=1}^N \(\x^i(T_1)-X^i(T_1)\)^2
+ \bar C\biggl(\frac{\hat C}{N^2}\biggr)^2
$$
for some constants $\bar c,\bar C>0$ depending only on $c_0$,
as desired.
\end{proof}

\subsection{The convergence results}

Combining the results in the previous sections,
we can now prove that if a continuous and a discrete solution
are close up to $1/N^2$ at time zero,
then they remain close for all time.
As one can see from the proof,
it is crucial that the discrete scheme has a error of order $\frac{1}{N^2}$
(see Lemma \ref{lem:taylor}).

\begin{thm}
\label{thm:conv rho}
Let $\x^i$ be a solution of the ODE \eqref{eq:xi},
and let $X^i$ be as in \eqref{eq:Xi}.
Assume that $X_0 \in C^{4,\alpha}([0,1])$, $X_0(0)=0$, $X_0(1)=1$, and that 
$a_0 \leq \partial_\theta X_0 \leq A_0$ for some positive constants $a_0,A_0$.
Also, suppose that
\begin{equation}
\label{eq:initialX}
|X^i(0) - \x^i(0)| \leq \frac{C'}{N^2} \qquad \forall\,i=1,\ldots,N.
\end{equation}
for some positive constant $C'$.

Then, there exists $\e_1\equiv \e_1\bigl(a_0, A_0, \|\rho\|_{C^{3,\alpha}([0,1])}, \|X_0\|_{C^{4,\alpha}([0,1])}\bigr)>0$ such that, if $\|\rho'\|_\infty+\|\rho''\|_\infty \leq \e_1$
we have
$$
\frac{1}{N} \sum_{i=1}^N \(\x^i(t)-X^i(t)\)^2
\leq \frac{\bar{\bar{C}}}{N^4}\qquad \forall\,t \in [0,\infty).
$$
\end{thm}

\begin{proof}
The idea of the proof is the following:
we want to prove the discrete gradient flow and the continuous one are $L^2$ close for all times. This is exactly what is claimed in Lemma \ref{lem:L2} which, on the other hand, is based on the assumption $\frac{c_0}{N}\leq \x^i(t) - \x^{i-1}(t)\leq \frac{C_0}{N},\ \ c_0, C_0 \in \br^+$.
Unfortunately, a priori, these assumptions may not hold for every time.
However, by carefully combining Lemmas \ref{lem:Linfty} 
and \ref{lem:L2} by an induction argument, we can show that these assumptions actually holds for all times. 

{\em Basis for the induction}. First we observe that, by Proposition \ref{prop:Lagr}, there exist two positive constants $a$ and $A$ such that
\be
\label{eq:Xt monotone}
a \leq \partial_\theta X(t) \leq A \qquad \forall\,t \geq 0.
\ee
Recalling the definition of $X^i$ in \eqref{eq:Xi}, we can infer the following inequalities at the discrete level:
\be
\label{eq:Xi monotone}
\frac{a}{N}\leq X^i(t) - X^{i-1}(t)\leq \frac{A}{N}\qquad \forall\,t \geq 0,\, \,\forall\,i.
\ee
Let us now focus on the assumption
$$
\frac{c_0}{N}\leq \x^i(t) - \x^{i-1}(t)\leq \frac{C_0}{N},\qquad c_0, C_0 \in \br^+.
$$
Using Lemma  \ref{lem:Linfty} we have
$$
|\x^i(t)-X^i(t)|\le |\x^i(0)-X^i(0)|+ \eta t \qquad \forall\,t \in [0,T].\\
$$
Keeping in mind the definition of $\eta$ and \eqref{eq:initialX}
we have
$$
|\x^i(t)-X^i(t)|\le\frac{C'}{N^2}+\frac{3\hat{C}}{N^2}t+\frac{C'}{N^2}\frac{t}{T} \qquad \forall\,t \in [0,T],
$$
so by the triangle inequality and \eqref{eq:Xi monotone} we obtain
$$
\frac{a}{N}-2\biggl(\frac{2C'}{N^2}+\frac{3\hat{C}}{N^2}T\biggr)
\leq \x^i(t) - \x^{i-1}(t)\leq \frac{A}{N}+2\biggl(\frac{2C'}{N^2}+\frac{3\hat{C}}{N^2}T\biggr)
$$
for $t \in [0,T]$.
In particular, by choosing $N$ large enough (depending only on $a,A,\hat C,C',T$),
we can ensure that
\be
\label{eq:x-x}
\frac{a}{2N}\leq \x^i(t) - \x^{i-1}(t)\leq \frac{2A}{N}\qquad \forall\,t \in [0,T].
\ee
{\em Inductive step}. Our goal is to show that if the above property holds for all $t \in [0,\alpha T]$ then it holds for all $t \in [0,(\alpha+1) T]$.
Let us apply Lemma \ref{lem:L2} on $[0,\alpha T]$ and  \eqref{eq:initialX} to get
\begin{align*}
\frac{1}{N} \sum_{i=1}^N \(\x^i(t)-X^i(t)\)^2&
\leq \frac{e^{-\bar c t} }{N} \sum_{i=1}^N \(\x^i(0)-X^i(0)\)^2
+ \bar C\biggl(\frac{\hat C}{N^2}\biggr)^2\\
&\le \frac{\bar{\bar{C}}}{N^4}\qquad \forall\,t \in [0,\alpha T]
\end{align*}
for some constant $\bar{\bar C}$ depending only on $\bar C,\hat C,C'$.
Hence, since
$$
|\x^i(t)-X^i(t)| \leq \sqrt{ \sum_{i=1}^N \big(\x^i(t)-X^i(t)\big)^2} \qquad
\forall\,t \in [0,\alpha T],\, \,\forall\,i,
$$
we obtain
in particular,
$$
|\x^i(\alpha T)-X^i(\alpha T)|\le \sqrt{\frac{\bar{\bar{C}}}{N^3}}\qquad
\forall\,i=1,\ldots,N.
$$
Applying again Lemma \ref{lem:Linfty} with $\alpha T$ as initial time, we now get
\begin{align*}
|\x^i(\alpha T+t)-X^i(\alpha T+t)|&\le |\x^i(\alpha T)-X^i(\alpha T)|+ \eta \alpha T \\
&\le \sqrt{\frac{\bar{\bar{C}}}{N^3}}+\frac{3\hat{C}}{N^2}\alpha T+\sqrt{\frac{\bar{\bar{C}}}{N^3}}\frac{t}{\alpha T} \qquad \forall\,t \in [0,\alpha T].
\end{align*}
Hence, by \eqref{eq:Xi monotone} and the triangle inequality,
\be
\label{eq:iterazione1}
\begin{split}
\frac{a}{N}-2\biggl(2\sqrt{\frac{\bar{\bar{C}}}{N^3}}-\frac{3\hat{C}}{N^2}\alpha T\biggr)&
\leq \x^i(t) - \x^{i-1}(t)\\
&\leq 
\frac{A}{N}+2\biggl(2\sqrt{\frac{\bar{\bar{C}}}{N^3}}-\frac{3\hat{C}}{N^2}\alpha T\biggr)
\end{split}
\ee
for each $t\in[\alpha T,(\alpha+1)T]$. Then, if $N$ is big enough so that
\be
\label{eq: N buono}
2\sqrt{\frac{\bar{\bar{C}}}{N^3}}+\frac{3\hat{C}}{N^2}\alpha T\le \frac{a}{4N}
\ee
we have
$$
\frac{a}{2N}\leq \x^i(t) - \x^{i-1}(t)\leq \frac{2A}{N}\qquad \forall\,t \in [\alpha T,(\alpha+1)T].
$$
Recalling the inequality $(\ref{eq:x-x})$ we get 
$$
\frac{a}{2N}\leq \x^i(t) - \x^{i-1}(t)\leq \frac{2A}{N}\qquad \forall\,t \in [0,(\alpha+1)T].
$$
This concludes the inductive step and, in particular, Lemma \ref{lem:L2} applied on $[0,\infty)$ proves the desired estimate
for $N \geq N_0$ for some large number $N_0$.

Notice that the case $N \leq N_0$ is trivial since (using that $0 \leq \x^i, X^i \leq 1$)
$$
\frac{1}{N} \sum_{i=1}^N \(\x^i(t)-X^i(t)\)^2
\leq 1 \leq 
 \frac{N_0^4}{N^4}\qquad \forall\,t \in [0,\infty).
$$

\end{proof}

\subsection{The Eulerian description}

In order to get a convergence result in Eulerian variable,
we will also need a full stability result in $L^2$ 
in the continuous case.
The following result holds:
\begin{prop}
\label{prop:L2 rho}
Assume that $\rho:[0,1]\to (0,\infty)$ is a periodic probability density of class $C^2$ and let $X_1, X_2$ be two solutions of the equation \eqref{eq:X} satisfying
\eqref{eq:boundary} and
\be
\label{mon_cond2}
0 <c_0 \leq \pt_\theta X_i(0,\theta)\le C_0,\qquad i=1,2.
\ee
There exists $\e_0\equiv \e_0 (c_0, C_0)$ as in Lemma \ref{lem:L2} such that, under the condition $\|\r'\|_{L^\infty}+ \|\r''\|_{L^\infty}\leq \e_0,$ one has

$$
 \int_0^1|X_1(t,\theta)-X_2(t,\theta)|^2\,d\te \le \( \int_0^1|X_1(0,\theta)-X_2(0,\theta)|^2\,d\te \)e^{-\bar c t}
$$
for all $t\ge 0$, for some $\bar c\equiv \bar c(c_0).$
\end{prop}
\begin{proof}
The proof of this result follows the lines of the proof of Proposition \ref{prop:L2 1},
with the difference that we have to get rid of the extra terms using the smallness 
of $\|\r'\|_{L^\infty}+ \|\r''\|_{L^\infty}$.
Also, this result could also be obtained as a consequence of Lemma \ref{lem:L2} letting $N \to \infty$.
However, since the proof is relatively short, we give it for the convenience of the reader.

We begin by noticing that since $\int_0^1 \rho(x)dx=1$, if $\|\rho'\|_\infty$ is sufficiently small it follows that $1/2 \leq \rho \leq 2$, so the monotonicity condition \eqref{mon_cond2} implies that
\be
\label{mon_cond3}
0 <c_1\leq \pt_\theta X_i(t)\le C_1,\qquad i=1,2,\, \mbox{for\ all}\,t \geq 0
\ee
for some constants $c_1,C_1$ depending only on $c_0,C_0$
(see Proposition \ref{prop:Lagr}).
Also, we notice that \eqref{eq:X} can be equivalently rewritten as
$$
\pt_t X=\frac14 \pt_\te\bigl(\rho(X)(\pt_\te X)^2\bigr) - \frac{1}{12}\rho'(X)(\pt_\te X)^3.
$$
Then, since $X_2-X_1$ vanishes at the boundary, we compute
\begin{align*}
&\frac{d}{dt} \int_0^1|X_1-X_2|^2\,d\te 
\\
&=\frac{1}{2}\int_0^1 (X_1-X_2)\,\Bigl( \pt_\te\bigl(\rho(X_1)(\pt_\te X_1)^2\bigr)- \pt_\te\bigl(\rho(X_2)(\pt_\te X_2)^2\bigr)\Bigr)\,d\te \\ 
&\qquad-\frac{1}6\int_0^1 (X_1-X_2)\,\Bigl(\rho'(X_1)(\pt_\te X_1)^3-\rho'(X_2)(\pt_\te X_2)^3\Bigr)\,d\te\\
&=-\frac{1}{2}\int_0^1 \pt_\te(X_1-X_2)\,\Bigl(\bigl(\rho(X_1)(\pt_\te X_1)^2\bigr)-\bigl(\rho(X_2)(\pt_\te X_2)^2\bigr)\Bigr)\,d\te \\ 
&\qquad-\frac{1}6\int_0^1 (X_1-X_2)\,\Bigl(\rho'(X_1)(\pt_\te X_1)^3-\rho'(X_2)(\pt_\te X_2)^3\Bigr)\,d\te\\
&=-\frac{1}{2}\int_0^1 \rho(X_2)\pt_\te(X_1-X_2)\,\bigl((\pt_\te X_1)^2-(\pt_\te X_2)^2\bigr)\,d\te \\ 
&\qquad-\frac{1}{2}\int_0^1 [\rho(X_1)-\rho(X_2)]\,\pt_\te(X_1-X_2)\,(\pt_\te X_1)^2\,d\te \\ 
&\qquad-\frac{1}6\int_0^1 \rho'(X_2) (X_1-X_2)\,\bigl((\pt_\te X_1)^3-(\pt_\te X_2)^3\bigr)\,d\te\\
&\qquad-\frac{1}6\int_0^1 [\rho'(X_1)-\rho'(X_2)]\,\rho (X_1-X_2)\,(\pt_\te X_1)^3\,d\te\\
&=:T_{1,1}+T_{1,2}+T_{2,1}+T_{2,2}.
\end{align*}
Recalling that $1/2 \leq \rho$, using  \eqref{mon_cond3} we get
\begin{align*}
T_{1,1} &\leq -\frac{1}{2}\int_0^1 \rho(X_2)\,\bigl(\pt_\te(X_1-X_2)\bigr)^2\,\bigl((\pt_\te X_1)+(\pt_\te X_2)\bigr)\,d\te\\
&\leq -\frac{c_1}{2}\int_0^1 \bigl(\pt_\te(X_1-X_2)\bigr)^2\,d\te.
\end{align*}
Using again  \eqref{mon_cond3} we bound
\begin{align*}
|T_{1,2}| &\leq \frac{C_1^2}{2} \|\rho'\|_\infty \int_0^1 |X_1 - X_2|\,|\pt_\te X_1 - \pt_\te X_2|\,d\te\\
&\leq \frac{C_1^2}{2} \|\rho'\|_\infty \int_0^1 (X_1 - X_2)^2\,d\te+
\frac{C_1^2}{2} \|\rho'\|_\infty \int_0^1 (\pt_\te X_1 - \pt_\te X_2)^2\,d\te,
\end{align*}
\begin{align*}
|T_{2,1}|&\leq \frac{C_1^2}{2} \|\rho'\|_\infty \int_0^1 |X_1 - X_2|\,|\pt_\te X_1 - \pt_\te X_2|\,d\te\\
&\leq \frac{C_1^2}{2} \|\rho'\|_\infty \int_0^1 (X_1 - X_2)^2\,d\te+
\frac{C_1^2}{2} \|\rho'\|_\infty \int_0^1 (\pt_\te X_1 - \pt_\te X_2)^2\,d\te,
\end{align*}
$$
|T_{2,2}|\leq \frac{C_1^3}{6} \|\rho''\|_\infty \int_0^1 (X_1 - X_2)^2\,d\te.
$$
Hence, combining all together, if both $\|\rho'\|_\infty$ and $\|\rho''\|_\infty $
are sufficiently small, using Poincar\'e inequality  (see Lemma \ref{lem:poincare}
and let $N \to \infty$),
 we obtain
\begin{align*}
\frac{d}{dt} \int_0^1|X_1-X_2|^2\,d\te 
&\leq  -\frac{c_1}{4}\int_0^1 \bigl(\pt_\te(X_1-X_2)\bigr)^2\,d\te\\
&\qquad+C\bigl(\|\rho'\|_\infty+\|\rho''\|_\infty\bigr)\int_0^1 (X_1 - X_2)^2\,d\te\\
&\leq  -\frac{c_1}{2}\int_0^1 (X_1-X_2)^2\,d\te\\
&\qquad+C\bigl(\|\rho'\|_\infty+\|\rho''\|_\infty\bigr)\int_0^1 (X_1 - X_2)^2\,d\te\\
&\leq -\frac{c_1}{4}\int_0^1 (X_1-X_2)^2\,d\te,
\end{align*}
and the result follows by Gronwall's inequality.
\end{proof}

\begin{thm}
Let $\rho:[0,1]\to (0,\infty)$ be a periodic probability density of class $C^{3,\alpha}$.
Let $x^i$ be a solution of the discrete gradient flow
starting from an initial datum satisfying
$$
\biggl|x^i(0)-X_0\biggl(0,\frac{i-1/2}N\biggr)\biggr| \leq \frac{C'}{N^2}\qquad \forall\,i=1,\ldots,N,
$$
where  $X_0 \in C^{4,\alpha}([0,1])$, $X_0(0)=1$, $X_0(1)=1,$  and $0 <c_0 \leq \partial_\theta X_0\leq C_0$.
Then there exist  three constants $\e_1\equiv \e _1\bigl(c_0, C_0, \|\rho\|_{C^{3,\alpha}([0,1])}, \|X_0\|_{C^{4,\alpha}([0,1])} \bigr)$ as in Theorem \ref{thm:conv rho}; $\bar{\bar c}\equiv \bar{\bar c}(c_0)>0, \bar{\bar C}\equiv \bar{\bar C}(c_0)>0$, 
such that,
$$
MK_1(\mu_t^N,\gamma \rho^{1/3}\,d\theta) \leq \bar{\bar C} \,e^{-\bar{\bar c}t/N^3}+\frac{\bar{\bar C}}{N} \qquad\forall\,t \geq 0,
$$
where 
$$
\gamma:=\frac{1}{\int_0^1 \rho^{1/3}(x)dx}
$$
provided that $\|\rho'\|_\infty+\|\rho''\|_\infty \leq \e_1.$
In particular
$$
MK_1(\mu_t^N,\gamma \rho^{1/3}\,d\theta) \leq \frac{2\bar{\bar C}}{N} \qquad \mbox{for\ all} \ \,t \geq \frac{N^3\log N}{\bar{\bar c}}.
$$
\end{thm}
\begin{proof}
Let $\bar X$ satisfy
$$
\partial_\theta \bar X=\frac{1}{\gamma \rho^{1/3}\circ \bar X},\quad \bar X(0)=0.
$$
Then $\bar X$ is a stationary solution of \eqref{eq:X} satisfying also the boundary condition \eqref{eq:boundary}, hence by
Proposition \ref{prop:L2 rho} we deduce that
$$
\int_0^1 |X(t)-\bar X|^2\,d\theta \leq C e^{-\bar c t},
$$
where $X(t)$ is the solution of \eqref{eq:X} starting from $X_0$.
We then apply Theorem \ref{thm:conv rho} to deduce that 
$$
\frac{1}{N} \sum_{i=1}^N \(x^i(t)-X^i(t/N^3)\)^2
\leq \frac{\bar{\bar{C}}}{N^4}\qquad \forall\,t \in [0,\infty),
$$
where $X^i(t):=X\left(t,\frac{i-1/2}N\right)$. Combining these two estimates and observing that
$\bar X_\# d\theta=\gamma\,\rho^{1/3}\,d\theta$, 
the result follows by arguing as in the proof of 
Theorem \ref{thm:Euler 1}.
\end{proof}

\appendix

\section{From the discrete to the continuous case}
\label{app:discrete cont}

In order to obtain a continuous version of the functional
$$
F_{N,r}(x^{1}, \ldots, x^{N})=\int_{0}^1\underset{1\le i \le N}{\mbox{min}} | x^i-y |^r\rho(y)\,dy,
$$ 
with $0\le x^{1}\le \ldots \le x^{N}\le1$,
we define
$$
x^{i+1/2}:=\frac{x^i+x^{i+1}}{2},
$$ 
where by convention $x^{0}=0$ and $x^{N+1}=1$.
Then the expression for the minimum becomes
$$
\min_{1\le j\le N}|y-x^{j}|^r=\left\{ \begin{array}{ccc}
|y-x^i|^r & \mbox{for} \ y \in (x^{i-1/2}, x^{i+1/2}),\\
|y|^r & \mbox{for} \ y \in (0, x^{1/2}),\\
|y-1|^r & \mbox{for} \ y \in (x^{N+1/2}, 1),\\
\end{array}
\right.
$$
and $F_{N,r}$ is given by
\begin{multline*}
F_{N,r}(x^{1}, \ldots, x^{N})=\\
 \sum_{i=1}^N \int_{x^{i-1/2}}^{x^{i+1/2}} |y-x^i|^r\r(y)dy+ \int_{0}^{x^{1/2}} |y|^r\r(y)dy
+\int_{x^{N+1/2}}^{1} |y-1|^r\r(y)dy.
\end{multline*}
Assume that 
$$
x^i=X\bigg(\frac{i-1/2}{N}\bigg), \qquad i=1, \ldots, N
$$
with $X: [0,1] \to [0,1]$ a smooth non-decreasing map. Then a Taylor expansion yields
$$
F_{N,r}(x^{1}, \ldots, x^{N})= \frac{C_r}{N^r}\int_0^1 \r(X(\te))|\pt_\te X(\te)|^{r+1}d\te+ {O}\Big(\frac{1}{N^{r+1}}\Big),
$$
where
$C_r=\frac{1}{2^r(r+1)}$ and ${O}\left(\frac{1}{N^{r+1}}\right)$ depends on the smoothness of $\rho$ and $X$
(for instance, $\rho \in C^1$ and $X \in C^2$ is enough).
Hence 
$$
N^rF_{N,r}(x^{1}, \ldots, x^{N}) \longrightarrow C_r\int_0^1 \r(X(\te))|\pt_\te X(\te)|^{r+1}d\te:=\mc{F}[X]
$$
as $N\ra \infty.$

\section{The Hessian of $\mathcal F[X]$}
\label{app:hessian}
Assume $\lambda \leq \rho \leq \frac{1}{\lambda}$, and
let $X,Y \in L^2([0,1])$ with $0\leq c \leq \partial_\theta X\leq C$ and
$|\partial_\theta Y|\leq C$.
Then
\begin{align*}
D^2\mathcal F[X](Y,Y)
&= 6\int_0^1 \rho(X)\,\pt_\te X\,(\pt_\te Y)^2\,d\te\\
& + 6 \int_0^1 \rho'(X)\,(\pt_\te X)^2\,(\pt_\te Y)\,Y\,d\te
+ \int_0^1 \rho''(X)\,(\pt_\te X)^3\,Y^2\,d\te.
\end{align*}
\subsection{Convexity under a smallness assumption on $\rho'$ and $\rho''$}
We want to prove that the Hessian of $\mathcal F$ is positive definite provided that
$$
 \|\rho'\|_\infty+ \|\rho''\|_\infty \ll 1.
 $$

We first observe that
\begin{multline*}
D^2\mathcal F[X](Y,Y)= \frac{d^2}{d^2\e}\bigg|_{\e=0}\mathcal F(X+\e Y) \geq 6\lambda\,c \int_0^1(\pt_\te Y)^2\,d\te\\
 - 6C^2 \|\rho'\|_\infty \int_0^1|\pt_\te Y|\,|Y|\,d\te
-C^3 \|\rho''\|_\infty \int_0^1 Y^2\,d\te.
\end{multline*}
Observe that if both $\rho'$ and $\rho''$ are small, we can control both the second and third term 
by the first one using
Cauchy-Schwarz and Poincar\'e inequalities.
In particular one sees that the Hessian is positive at ``points'' $X$ which are uniformly monotone and Lipschitz \footnote{Recall that $0\leq c \leq \partial_\theta X\leq C$}.

Indeed, using Cauchy-Schwarz,
\begin{align*}
D^2\mathcal F[X](Y,Y)& \geq -C^3 \|\rho''\|_\infty \int_0^1 Y^2\,d\te+6\lambda\,c \int_0^1(\pt_\te Y)^2\,d\te\\
&\qquad -3C^2 \|\rho'\|_\infty \bigg[ \int_0^1 Y^2\,d\te+ \int_0^1 (\pt_\te Y)^2\,d\te \bigg].\\
\end{align*}
Hence, if $3\|\rho'\|C^2\le 3\lambda c$ we have
\begin{align*}
D^2\mathcal F[X](Y,Y)&\geq  3\lambda\,c \int_0^1(\pt_\te Y)^2\,d\te\\
&\qquad +\bigg[-C^3 \|\rho''\|_\infty +3C^2  \|\rho'\|_\infty\bigg] \int_0^1 Y^2\,d\te\\
& \ge 6\lambda c \int_0^1 Y^2\,d\te +\bigg[C^3 \|\rho''\|_\infty -3C^2  \|\rho'\|_\infty\bigg] \int_0^1 Y^2\,d\te,
\end{align*}
where for the second inequality we used Poincar\'e (see for instance Lemma \ref{lem:poincare}
and let $N \to \infty$).
Thus, if $3\lambda c > C^3 \|\rho''\|_\infty -3C^2  \|\rho'\|_\infty$ it follows that the Hessian of $\mathcal F$
is positive definite.

\subsection{Lack of convexity without a smallness assumption}
In this section it will be convenient to specify the dependence of $\mathcal F$ on $\rho$, so we denote
$$
\mathcal F_\rho(X):=\int_0^1 \rho(X)\,|\partial_\theta X|^3\,d\theta.
$$
To build a counterexample, we first construct a density $\bar \rho$ and a Lipschitz function $Y$
such that $D^2\mathcal F_{\bar \rho}(X)[Y,Y]<0$ with $X(t, \te)=\te$.
Although the density $\bar\rho$ will not be smooth nor strictly positive, by an approximation
argument we will eventually obtain a counterexample also with a smooth positive density.

Let $\e>0$ be a small number and define
$$
\bar{\rho}(\te):= \left\{ \begin{array}{cc} 1 & \mbox{for}\ \te \in \[\frac{1}{2}-\e,  \frac{1}{2}+\e\]\\
0 & \mbox{for}\ \te \in [0,1] \setminus \[\frac{1}{2}-\e,  \frac{1}{2}+\e\],
\end{array} \right.
$$
and let $Y(t, \te)$ be a Lipschitz function in $[0,1]$ that coincides with $|\te-\frac{1}{2}|+1$ in  $\[\frac{1}{2}-\e,  \frac{1}{2}+\e\].$ Then, recalling the formula for the Hessian of $\mathcal F$,
\begin{align*}
D^2\mathcal F_{\bar\rho}(X)[Y,Y]= 6\int_0^1\bar \rho\,(\pt_\te Y)^2\,d\te
+6\int_0^1\bar \rho'\,\pt_\te Y\,Y\,d\te+\int_0^1\bar \rho''\,Y^2\,d\te.
\end{align*}
Integrating by parts we have
\begin{align*}
D^2\mathcal F_{\bar\rho}(X)[Y,Y]& =6\int_0^1\bar \rho\,(\pt_\te Y)^2\,d\te-6\int_0^1\bar \rho\,(\pt_\te Y)^2-6\int_0^1\bar \rho\, \pt^2_\te Y\, Y\, d\te\\
&\qquad+2\int_0^1\bar \rho\bigg[(\pt_\te Y)^2+\pt^2_\te Y\, Y  \bigg]\, d\te\\
&=2\int_0^1\bar \rho\,(\pt_\te Y)^2\,d\te-4\int_0^1 \bar \rho\,\pt^2_\te Y\, Y\, d\te.
\end{align*}
Recalling the definitions of $Y$ and $\bar \rho$ we have
\begin{align*}
D^2\mathcal F_{\bar\rho}(X)[Y,Y]& =2\int_{\frac{1}{2}-\e}^{\frac{1}{2}+\e}(\pt_\te Y)^2\,d\te-4Y\(\frac{1}{2}\)=4\e-4< 0 \qquad \mbox{for} \  \e<1.\\
\end{align*}
In order to build a counterexample with a smooth positive density, we first extend $\bar \rho$ by periodicity on the whole real line and define $\rho_{\delta}:= \bar \rho\, *\, \vp_{\delta},$ with 
$$
\vp_{\delta}(\te)= \frac{\exp ^{-\frac{|\te|^2}{2\delta}}}{\sqrt{2\pi \delta}}.
$$
Then, by the same computation as above we have
\begin{align*}
D^2\mathcal F_{\rho_\delta}(X)[Y,Y]& =2\int_0^1 \rho_{\delta}\,(\pt_\te Y)^2\,d\te-4\int_0^1 \rho_{\delta}\,\pt^2_\te Y\, Y\, d\te,
\end{align*}
and since $\rho_\delta \to \bar \rho$ in $L^1$ and $\rho_\delta(1/2) \to \bar \rho(1/2)$, we conclude that
$$
D^2\mathcal F_{\rho_\delta}(X)[Y,Y] \to D^2\mathcal F_{\bar \rho}(X)[Y,Y] \mbox{as} \ \delta \to 0.
$$
In particular, by choosing $\delta>0$ sufficiently small, we have obtained that the Hessian of $\mathcal F_{\rho_\delta}$ in the direction $Y$ is negative when $X(\te)=\te$ and $\rho_{\delta}\in C^{\infty}([0,1])$ and satisfies $1\ge \rho_{\delta}>0$.

\bigskip

{\it Acknowledgments:} The third author is grateful to Matteo Bonforte for useful comments on a preliminary version of this paper.


\end{document}